\documentclass[11pt]{article}
\usepackage{amsthm, amsmath, amssymb, amsfonts, url, booktabs, tikz, setspace, fancyhdr, bm}
\usepackage{hyperref}
\usepackage{amsthm}
\usepackage{geometry}
\usepackage{hyperref, enumerate}
\usepackage[shortlabels]{enumitem}
\usepackage[babel]{microtype}
\usepackage[english]{babel}
\usepackage[capitalise]{cleveref}
\usepackage{comment}
\usepackage{bbm}
\usepackage{csquotes}
\usepackage{mathabx}
\usepackage{tikz}
\usetikzlibrary{calc, angles, quotes, intersections}
\usepackage{graphicx}
\usepackage{float}
\usepackage{xcolor}

\counterwithin{figure}{section}

\newtheorem{conjecture}{Conjecture}

\newtheorem{theorem}{Theorem}[section]
\newtheorem{prop}[theorem]{Proposition}
\newtheorem{lemma}[theorem]{Lemma}
\newtheorem{cor}[theorem]{Corollary}

\newtheorem{claim}[theorem]{Claim}

\theoremstyle{definition}
\newtheorem{defn}[theorem]{Definition}
\newtheorem*{defn-non}{Definition}

\newtheorem{ques}[theorem]{Question}

\newtheorem*{rmk}{Remark}

\newlist{Case}{enumerate}{2}
\setlist[Case, 1]{%
    label           =   {\bfseries Case \arabic*.},
    labelindent=1em ,labelwidth=1.3cm, labelsep*=1em, leftmargin =!
}
\setlist[Case, 2]{%
    label           =   {\bfseries Subcase \arabic{Casei}.\arabic*.},
    labelindent=-1em ,labelwidth=1.3cm, labelsep*=1em, leftmargin =!
}

\newenvironment{poc}{\begin{proof}[Proof of claim]}{\end{proof}}

\newcommand{\ceil}[1]{\lceil #1\rceil}

\usepackage{todonotes}

\title{Canonical Ramsey: triangles, rectangles and beyond}

\author{
Yijia Fang\thanks{Department of Mathematics, National University of Singapore, Singapore. Email: fangyijia@u.nus.edu.}
\and
Gennian Ge\thanks{School of Mathematical Sciences, Capital Normal University, Beijing 100048, China. Email: gnge@zju.edu.cn. Gennian Ge is supported by the National Key Research and Development Program of China under Grant 2020YFA0712100, the National Natural Science Foundation of China under Grant 12231014, and Beijing Scholars Program.}
\and 
Yang Shu\thanks{School of Mathematical Sciences, University of Science and Technology of China, Hefei, Anhui 230026, China. Emails: shuyangyyyy@mail.ustc.edu.cn, dilongyang@mail.ustc.edu.cn.}
\and
Qian Xu\thanks{School of Mathematical Sciences, Peking University, Beijing 100871, China. Email: qianbrianxu@stu.pku.edu.cn. }
\and
Zixiang Xu\thanks{Extremal Combinatorics and Probability Group (ECOPRO), Institute for Basic Science (IBS), Daejeon 34126, South Korea. Email: zixiangxu@ibs.re.kr. Supported by IBS-R029-C4.}
\and
Dilong Yang \footnotemark[3]
}

\date{}

\begin{document}
\maketitle

\begin{abstract}
In a seminal work, Cheng and Xu showed that if \(\mathcal{S}\) is a square or a triangle with a certain property, then for every positive integer \(r\) there exists \(n_0(\mathcal{S})\) independent of \(r\) such that every \(r\)-coloring of \(\mathbb{E}^n\) with \(n\ge n_0(\mathcal{S})\) contains a monochromatic or a rainbow congruent copy of \(\mathcal{S}\). Geh\'{e}r, Sagdeev, and T\'{o}th formalized this dimension independence as the {canonical Ramsey property} and proved it for all hypercubes, thereby covering rectangles whose squared aspect ratio \((a/b)^2\) is rational. They asked whether this property holds for all triangles and for all rectangles.
 
\begin{enumerate}
    \item  We resolve both questions. More precisely, for triangles we confirm the property in \(\mathbb{E}^4\) by developing a novel rotation-sphereical chaining argument. For rectangles, we introduce a structural reduction to product configurations of bounded color complexity, enabling the use of the simplex Ramsey theorem together with product Ramsey theorem.
    \item Beyond this, we develop a concise perturbation framework based on an iterative embedding coupled with the Frankl–R\"{o}dl simplex super-Ramsey theorem, which yields the canonical Ramsey property for a natural class of \(3\)-dimensional simplices and also furnishes an alternative proof for triangles.

\end{enumerate}
\end{abstract}

\section{Introduction}
\subsection{Background}
Euclidean Ramsey theory, initiated in the 1970s by Erd\H{o}s, Graham, Montgomery, Rothschild, Spencer, and Straus~\cite{1973JCTA}, asks for which finite configurations \(\mathcal{S}\) every finite coloring of a sufficiently high-dimensional Euclidean space must contain a monochromatic congruent copy of \(\mathcal{S}\).
Beyond its concise formulation, the subject provides a natural setting for understanding how combinatorial constraints from finite colorings enforce geometric patterns under the isometries of \(\mathbb{E}^n\). Over the past five decades this line of inquiry has fostered sustained interaction between combinatorics and geometry; see~\cite{2019DCGCONLON,2022ConlonNonspherical,conlon2022more,2024Yip,2026JCTA,1990JAMS,2004FranklRodl,2025EUJC,1980GrahamJCTA,2009Combinatorica,2007DCGAnote,SHADER1976385} and the references therein.
Methods range from classical Ramsey and extremal techniques to probabilistic and geometric arguments, and the problems often interface with stability phenomena, transference principles, and the construction of color-avoiding obstructions.

In parallel, Gallai–Ramsey theory studies the dichotomy ``monochromatic vs rainbow'' in colored structures~\cite{2020SIDMABalogh,1967Gallai,2010JGTGyarfas,2024LIDM,2020JGTLiu,2022JGT,2023JCTAMao,2023EuJC}.
Merging these directions, Mao, Ozeki, and Wang~\cite{2022arxivEGR} proposed the \emph{Euclidean Gallai–Ramsey} paradigm: for configurations \(\mathcal{K}_1,\mathcal{K}_2\) and a positive integer \(r\),
\[
\mathbb{E}^n \overset{r}{\rightarrow} (\mathcal{K}_1;\mathcal{K}_2)_{\mathrm{GR}}
\]
to mean that every coloring \(\chi:\mathbb{E}^n\to[r]\) yields a monochromatic copy of \(\mathcal{K}_1\) or a rainbow copy of \(\mathcal{K}_2\).
This framework packages a canonical alternative that refines classical existence questions within a unified Euclidean setting.

\subsection{Known results and two central problems of Gehe\'{e}r, Sagdeev and T\'{o}th}
Cheng and Xu~\cite{2025DCGChengXu} studied this problem for various fundamental configurations, with particular focus on triangles, squares, and lines. They discovered a \emph{dimension-independence phenomenon}: once the ambient dimension exceeds a threshold depending only on the configuration, monochromatic or rainbow copies are unavoidable. For instance, for any right or acute triangle, or any square \(\mathcal{S}\), there exists an integer \(n_0\), independent of \(r\), such that \(\mathbb{E}^n \overset{r}{\rightarrow} (\mathcal{S}; \mathcal{S})_{\mathrm{GR}}\) holds for all \(n \ge n_0\); see also their canonical results for certain simplices~\cite[Corollaries~1.5 and 1.7]{2025DCGChengXu}.

To formalize this, Gehe\'{e}r, Sagdeev, and T\'{o}th~\cite{2024Cano} introduced the \emph{canonical Ramsey property}: a finite configuration \(\mathcal{X}\) exhibits this property if there exists \(n_0=n_0(\mathcal{X})\) such that, for any positive integer \(r\) and any \(n\ge n_{0}\),
\[
\mathbb{E}^{n} \overset{r}{\rightarrow} (\mathcal{X}; \mathcal{X})_{\mathrm{GR}}.
\]
The authors in~\cite{2024Cano} established this for all hypercubes and thereby for rectangles with side lengths \(a,b\) such that \((a/b)^2\) is rational, and they improved the acute-triangle case of~\cite{2025DCGChengXu} by showing that for any fixed acute triangle \(\mathcal{T}\) one can take \(n_0(\mathcal{T})=3\).

Despite these advances, two problems were explicitly highlighted in~\cite{2024Cano} as the main outstanding challenges in the theory.

\begin{ques}[Gehe\'{e}r-Sagdeev-T\'{o}th~\cite{2024Cano}]\
\begin{enumerate}
    \item[\textup{(1)}] Do all triangles have the canonical Ramsey property? 
    \item[\textup{(2)}] Do all rectangles have the canonical Ramsey property?
\end{enumerate}
\end{ques}

\subsection{Our contributions}
We settle both questions of Gehe\'er-Sagdeev-T\'{o}th in the affirmative.

\begin{theorem}\label{thm:triangle}
    Let \(\mathcal{T}\) be a triangle and let \(r\) be a positive integer. Then
    \[
    \mathbb{E}^{4}\overset{r}{\rightarrow}(\mathcal{T};\mathcal{T})_{\mathrm{GR}}.
    \]
\end{theorem}

\begin{theorem}\label{thm:rectangle}
Let \(r\) be a positive integer and let \(\mathcal{R}\) be a rectangle. There exists an integer \(n_0 = n_0(\mathcal{R})\) such that for all \(n \ge n_0\),
\[
    \mathbb{E}^{n} \overset{r}{\rightarrow} (\mathcal{R}; \mathcal{R})_{\mathrm{GR}}.
\]
\end{theorem}

Theorem~\ref{thm:triangle} shows that four dimensions already suffice for every triangle, sharpening the general expectation and subsuming the previously treated acute and right cases. Our proof creates a rotation–spherical chaining argument which manufactures the desired copy from an assumed obstruction. Theorem~\ref{thm:rectangle} treats all aspect ratios. Here the key input is a bounded color complexity reduction that extracts an auxiliary product structure inside \(\mathbb{E}^{n}\); this unlocks a combination of the Frankl–R\"odl simplex Ramsey theorem~\cite{1990JAMS} with product Ramsey property~\cite{1973JCTA} to pass from local control to the global configuration.

Beyond triangles and rectangles, we develop a general high-dimensional mechanism that we expect to be of independent interest.
At its core lies a \emph{canonical perturbation} step that, in a product ambient space, selects a single tuple of auxiliary fiber points enforcing a uniform local non-monochromaticity on prescribed \(\varepsilon\)-spheres.
This principle is powered by the simplex super-Ramsey theorem of Frankl--R\"{o}dl~\cite{1990JAMS} and underpins our subsequent constructions. For clarity we record the ambient notion of \emph{\(\varepsilon\)-sphere} we use. We usually write \(\|\cdot\|\) for the Euclidean distance.

\begin{defn}[\(\varepsilon\)-sphere]\label{def:eps-sphere}
Fix integers \(m\ge 1\) and \(n_0,\dots,n_m\ge 1\) and identify the ambient space
\(\mathbb{E}^{n_0}\times\cdots\times\mathbb{E}^{n_m}\).
For a point \(\boldsymbol{y}=(\boldsymbol{y}_0,\dots,\boldsymbol{y}_m)\) with \(\boldsymbol{y}_i\in\mathbb{E}^{n_i}\) and an index \(j\in[m]\), the \emph{\(\varepsilon\)-sphere of \(\boldsymbol{y}\) in \(\mathbb{E}^{n_j}\)} is
\[
S_j(\boldsymbol{y};\varepsilon)
:=\Bigl\{(\boldsymbol{y}_0,\dots,\boldsymbol{y}_{j-1},\boldsymbol{z},\boldsymbol{y}_{j+1},\dots,\boldsymbol{y}_m):
\ \boldsymbol{z}\in\mathbb{E}^{n_j},\ \|\boldsymbol{z}-\boldsymbol{y}_j\|=\varepsilon\Bigr\},
\]
that is, the usual Euclidean sphere of radius \(\varepsilon\) in the \(j\)-th fiber with all other coordinates fixed.
\end{defn}

We now state our canonical perturbation theorem, which asserts that, assuming there is no monochromatic copy of a fixed simplex \(\Delta\), one can choose a single tuple of fiber points so that every prescribed \(\varepsilon\)-sphere around each anchor is necessarily multicolored. 

\begin{theorem}\label{thm:canonical-perturbation}
Let \(d,n_0,m,r\ge 1\) be integers and let \(\Delta\subseteq\mathbb{E}^{d}\) be a simplex.
Let \(\mathcal{X}=\{\boldsymbol{v}_1,\dots,\boldsymbol{v}_m\}\subseteq \mathbb{P}_0\cong\mathbb{E}^{n_0}\) be any \(m\)-point set, and fix \(\varepsilon>0\).
Then there exist integers \(n_1,\dots,n_m\) with the following property: for every \(r\)-coloring \(\chi\) of the orthogonal product
\[
\mathbb{P}:=\mathbb{P}_0\times\mathbb{P}_1\times\cdots\times\mathbb{P}_m,
\qquad \mathbb{P}_i\cong\mathbb{E}^{n_i}\ (i=0,1,\dots,m),
\]
that contains no monochromatic copy of \(\Delta\), there exist points \(\boldsymbol{x}_i\in\mathbb{P}_i\) for \(i=1,\dots,m\) such that, writing
\(\boldsymbol{p}_i:=(\boldsymbol{v}_i,\boldsymbol{x}_1,\dots,\boldsymbol{x}_m)\in\mathbb{P}\),
each \(\varepsilon\)-sphere \(S_i(\boldsymbol{p}_i;\varepsilon)\) is not monochromatic under \(\chi\).
\end{theorem}
This uniform local obstruction later allows us to recover the prescribed distances and run the color-structure analysis that produces a monochromatic or a rainbow copy of \(\Delta\). Using this, we can provide a new and short proof of canonical Ramsey property for all triangles, see~\cref{subsection:NewproofTriangle}. Moreover, we can show the canonical Ramsey property to certain \(3\)-dimensional simplices.

\begin{theorem}\label{thm:3-dimensionalSimplex}
   Let \(r\) be a positive integer. Let \(\mathcal{S}\) be a \(3\)-dimensional simplex such that at least one of the heights of \(\mathcal{S}\) is greater than the circumradius of one of its four faces. Then there exists some \(n_{0}(\mathcal{S})\) such that for any \(n\ge n_{0}(\mathcal{S})\)
    \[
    \mathbb{E}^{n} \overset{r}{\rightarrow} (\mathcal{S}; \mathcal{S})_{\mathrm{GR}}.
\]
\end{theorem}
\cref{thm:3-dimensionalSimplex} substantially strengthens the earlier \(3\)-dimensional simplices result \cite[Corollary~1.7]{2025DCGChengXu}: while Cheng and Xu assumed that {all four} heights exceed the circumradius of the target simplex, our theorem requires only one height to exceed the circumradius of some face. 

We prove the three results in separate sections and, for clarity, allow slightly different notation; each section records its standing notation at the outset.

\section{Triangles: Proof of~\cref{thm:triangle}}
By the results of~\cite{2025DCGChengXu,2024Cano}, it suffices to treat the case in which \(\mathcal{T}\) is an obtuse triangle with side lengths \(a\le b<c\) satisfying \(a+b>c\) and \(a^{2}+b^{2}<c^{2}\). Moreover, \cite{1975EGMB} shows that every \(2\)-coloring of \(\mathbb{E}^{3}\) already contains a monochromatic copy of \(\mathcal{T}\); hence we may assume \(r\ge 3\). We begin by recording the necessary definitions, notation, and auxiliary facts, and then present the full proof.

\subsection{Notation and auxiliary results}
Throughout the proof, since this proof mainly relies on geometric argument, we use \(|AB|\) for the Euclidean distance between points \(A\) and \(B\). For a point \(A\) and a segment \(BC\), we write \(d(A,BC)\) for the Euclidean distance from \(A\) to the segment \(BC\). We write \(\triangle ABC\) for the triangle with vertices \(A,B,C\). Setting \(a=|BC|\), \(b=|CA|\), and \(c=|AB|\), let \(\operatorname{Area}(\triangle ABC)\) denote its area, Heron’s formula reads
\[
16\cdot\operatorname{Area}(\triangle ABC)^2
=2a^2b^2+2b^2c^2+2c^2a^2-a^4-b^4-c^4.
\]
The notation \(AB \perp CD\) indicates that the segments \(AB\) and \(CD\) are perpendicular. For a sphere \(\Gamma\) we write \(\operatorname{rad}(\Gamma)\) for its radius. For a (nondegenerate) simplex \(\mathcal{S}\), \(\operatorname{rad}(\mathcal{S})\) denotes its circumradius, that is, the radius of the smallest sphere containing \(\mathcal{S}\). For a set \(\mathcal{F}\subseteq \mathbb{E}^n\), we write \(\operatorname{aff}(\mathcal{F})\) for the affine subspace induced by \(\mathcal{F}\), that is,
\[
\operatorname{aff}(\mathcal{F})
=\Bigl\{\ \lambda_1 \boldsymbol{x}_1+\cdots+\lambda_m \boldsymbol{x}_m\ \Bigm|\ 
m\in\mathbb{N},\ \boldsymbol{x}_i\in\mathcal{F},\ \lambda_i\in\mathbb{R},\ \sum_{i=1}^m \lambda_i=1\ \Bigr\}.
\]
For \(N\in\mathbb{N}\) and \(s>0\), an \emph{\(N\)-sphere} in \(\mathbb{E}^m\) with \(m\ge N+1\) is the set
\[
\Gamma_N(s;O,\Pi)\;:=\;\{\,A\in \Pi:\ |AO|=s\,\},
\]
where \(\Pi\subseteq\mathbb{E}^m\) is an affine \((N+1)\)-plane and \(O\in\Pi\) is the center. Its intrinsic dimension is \(N\), its radius is \(s\), and \(O\) is its center. When the ambient data \((O,\Pi)\) are clear from context, we simply write \(\Gamma_N(s)\).

For an obtuse triangle with side lengths \(a,b,c\) with \(a\le b<c\), we define
\[
\Delta:= -a^{4}-b^{4}-c^{4}+2a^{2}b^{2}+2b^{2}c^{2}+2c^{2}a^{2}\ (>0),
\quad
h(a,b,c):=\frac{\sqrt{\Delta}}{a}.
\]
For \(0<\varepsilon<h(a,b,c)\), set
\[
\ell_{a,b,c}(\varepsilon):=
\sqrt{\frac{\Delta-a^{2}\varepsilon^{2}}{\,b^{2}-(\varepsilon/2)^{2}\,}}.
\]
We then build the following inequalities.
\begin{prop}\label{prop:KeyInequality}\
\begin{enumerate}
\item[\textup{(1)}] \(h(a,b,c)<2b\).
\item[\textup{(2)}] If \(0<\varepsilon<h(a,b,c)\), then
\(\ \ell_{a,b,c}(\varepsilon)<2\sqrt{\,c^{2}-(\varepsilon/2)^{2}\,}\).
\end{enumerate}
\end{prop}

\begin{proof}[Proof of~\cref{prop:KeyInequality}]
Since \(\mathcal{T}\) is obtuse, \(a^{2}+b^{2}<c^{2}\). Using
\[
4a^{2}b^{2}-\Delta=(a^{2}+b^{2}-c^{2})^{2}>0,
\]
we obtain \(\sqrt{\Delta}<2ab\), hence \(h(a,b,c)=\frac{\sqrt{\Delta}}{a}<2b\), proving \textup{(1)}. In particular, if \(0<\varepsilon<h(a,b,c)\), then
\(b^{2}-(\varepsilon/2)^{2}>0\) and \(\Delta-a^{2}\varepsilon^{2}>0\), so \(\ell_{a,b,c}(\varepsilon)\) is well-defined.

For \textup{(2)}, note first that from \textup{(1)} we have \(h(a,b,c)<2b\), so \(0<\varepsilon<h(a,b,c)\) implies \(b^{2}-(\varepsilon/2)^{2}>0\). Since
\[
\ell_{a,b,c}(\varepsilon)
= \sqrt{\frac{\Delta-a^{2}\varepsilon^{2}}{\,b^{2}-(\varepsilon/2)^{2}\,}},
\]
squaring both sides of the desired inequality (which is allowed as both sides are positive) is equivalent to
\[
\frac{\Delta-a^{2}\varepsilon^{2}}{\,b^{2}-(\varepsilon/2)^{2}\,}
\;<\;4\Bigl(c^{2}-\frac{\varepsilon^{2}}{4}\Bigr)
\;=\;4c^{2}-\varepsilon^{2}.
\]
Rearranging gives
\[
0<4b^{2}c^{2}-\Delta+(a^{2}-b^{2}-c^{2})\varepsilon^{2}+\frac{\varepsilon^{4}}{4}
=(b^{2}+c^{2}-a^{2})^{2}-(b^{2}+c^{2}-a^{2})\varepsilon^{2}+\frac{\varepsilon^{4}}{4}
=\Bigl(b^{2}+c^{2}-a^{2}-\frac{\varepsilon^{2}}{2}\Bigr)^{2}.
\]
This is strictly positive whenever \(\varepsilon^{2}\ne 2(b^{2}+c^{2}-a^{2})\). Noticing that $c^2-a^2 \geq b^2$, our assumption
\(0<\varepsilon<h(a,b,c)\) implies
\[
\varepsilon^{2}
<h(a,b,c)^{2}
=\frac{\Delta}{a^{2}}
<(2b)^{2}
=4b^{2}
\le 2(b^{2}+c^{2}-a^{2}),
\]
so the inequality is strict, completing the proof of \textup{(2)}.
\end{proof}

We will also use the following result from~\cite[Proposition~2.2]{2025DCGChengXu}.
\begin{prop}[{\cite{2025DCGChengXu}}]\label{prop:TwoPoint}
Let \(n\ge 2\) and \(r\) be positive integers. For any configuration \(\mathcal{X}\subseteq \mathbb{E}^n\) and any two-point set \(\mathcal{P}\subseteq \mathbb{E}^n\), we have
\(
   \mathbb{E}^n \overset{r}{\rightarrow} (\mathcal{X};\mathcal{P})_{\textup{GR}}.
\)
\end{prop}

\subsection{Overview of the proof of~\cref{thm:triangle}}
Although the outline of the proof is clear, we simply sketch the proof. We argue by contradiction and may assume the triangle is obtuse and the number of colors is at least three. The quantitative parameters \(h(a,b,c)\) and \(\ell_{a,b,c}(\varepsilon)\) are introduced to control a small base edge \(AB\) and the induced geometry on the sphere of points at distance \(c\) from both \(A\) and \(B\). Proposition~\ref{prop:TwoPoint} guarantees the existence of two differently colored points \(A,B\) at any prescribed small distance \(\varepsilon\). Proposition~\ref{prop:KeyInequality} ensures that for such \(\varepsilon\) the auxiliary chord length \(\ell_{a,b,c}(\varepsilon)\) is strictly below the relevant diameter, so all constructions below are valid.

Lemma~\ref{lem:5-point structure} is the rigidity step: if \(A\) and \(B\) have different colors, then for any two points \(P,M\) on the equal–distance sphere (centered at the midpoint of \(AB\)) with \(|PM|=\ell_{a,b,c}(\varepsilon)\), the colors of \(P\) and \(M\) must coincide. Proposition~\ref{lem: polygonal chain} supplies a polygonal chain on a sphere with all edges of the same prescribed length; chaining Lemma~\ref{lem:5-point structure} along such a path gives Lemma~\ref{cor:monochromatic sphere}: the entire equal distance sphere is monochromatic.

We then split by the circumradius of the fixed triangle. If it is smaller than \(c\), the monochromatic sphere contains a circle of the right radius and hence a monochromatic copy of the triangle. Otherwise, a rotation argument around a monochromatic sphere produces a concentric non-monochromatic sphere; a minimal-radius red shell yields a close red–nonred pair \(P,Q\) with controlled separation. Projecting to their perpendicular bisector plane furnishes points on the two spheres that are equidistant from \(P\) and \(Q\). Applying Lemma~\ref{cor:monochromatic sphere} to \(P,Q\) forces these two points to share a color, contradicting the setup. Thus the theorem holds.

\subsection{Two important lemmas}
\begin{lemma}\label{lem:5-point structure}
Let \(0<\varepsilon<h(a,b,c)\), and let \(A,B,P,M\) be four distinct points such that
\[
|AB|=\varepsilon,\quad |AM| = |BM| = |AP| = |BP| = c,\quad
|PM|=\ell_{a,b,c}(\varepsilon)<2\sqrt{\,c^2-(\varepsilon/2)^2\,}.
\]
Let \(r\ge 3\) and let \(\chi \colon \mathbb{E}^4 \to [r]\) be an \(r\)-coloring containing neither a monochromatic copy nor a rainbow copy of \(\mathcal{T}\). If \(\chi(A) \neq \chi(B)\), then \(\chi(M) = \chi(P)\).
\end{lemma}

\begin{proof}[Proof of~\cref{lem:5-point structure}]
All points below lie in the \(3\)-dimensional affine subspace \(\textup{aff}\{A,B,P,M\}\subseteq\mathbb{E}^4\). The key ingredient is to find the fifth point with desired properties.

\begin{claim}\label{prop:existence5points}
There exists a point \(N\) such that
\[
|NA|=|NB|=b,\qquad |NP|=|NM|=a.
\]
\end{claim}

\begin{poc}
    Let \(K\) and \(K'\) be the midpoints of \(AB\) and \(PM\), respectively. Since \(|PA|=|PB|=|MA|=|MB|=c\), both \(P\) and \(M\) lie on the perpendicular bisector of \(AB\); hence \(PK\perp AB\) and \(MK\perp AB\). In particular,
\[
|PK|=|MK|=\sqrt{\,c^2-(\varepsilon/2)^2\,}=:c'.
\]
Let \(b':=\sqrt{\,b^2-(\varepsilon/2)^2\,}\). Consider the auxiliary triangle \(\mathcal{T}'\) with side lengths \((a,b',c')\). Since \(0<\varepsilon<h(a,b,c)<2b<2c\), we have \(b',c'>0\). A direct Heron computation gives
\[
16\,\textup{Area}(\mathcal{T}')^2=\Delta-a^2\varepsilon^2>0,
\]
so this triangle is nondegenerate. Moreover, from \(c^2>a^2+b^2\) it follows that
\[
(c')^2=c^2-(\varepsilon/2)^2>a^2+b^2-(\varepsilon/2)^2=a^2+(b')^2,
\]
hence it is obtuse with the obtuse angle opposite \(c'\). Let
\[
h:=\frac{2\,\textup{Area}(\mathcal{T}')}{b'}=\frac12\sqrt{\frac{\Delta-a^2\varepsilon^2}{\,b^2-(\varepsilon/2)^2\,}}
=\frac{\ell_{a,b,c}(\varepsilon)}{2}.
\]
In the circle (in the plane perpendicular to \(AB\)) of center \(K\) and radius \(c'\) passing through \(P\) and \(M\), the chord \(PM\) has length \(\ell_{a,b,c}(\varepsilon)\); thus the line \(KK'\) is perpendicular to \(PM\) and
\[
|KK'|=\sqrt{\,{c'}^{2}-\bigl(\ell_{a,b,c}(\varepsilon)/2\bigr)^{2}\,}
=\sqrt{\,{c'}^{2}-h^{2}\,}.
\]
Put \(\mathcal{L}:=KK'\). The distance from \(P\) to the line segment \(\mathcal{L}\) equals \(|PK'|=\frac{\ell_{a,b,c}(\varepsilon)}{2}=h\), since \(PK'\perp \mathcal{L}\).
\begin{figure}[h]
\centering
\begin{tikzpicture}[scale=1.5]
    \coordinate (A) at (0,0,1);
    \coordinate (K) at (0.2,0,0);
    \coordinate (B) at (0.4,0,-1);
    \coordinate (N) at (0,1.5,0);
    \coordinate (P) at (-2,2,0);
    \coordinate (M) at (2,2,0);
    \coordinate (K') at (-0.1,2,0);

    \foreach \point/\position in {A/below,B/below,N/above right,P/left,M/right,K/below, K'/above} {
        \fill (\point) circle (0pt);
        \node[\position=2pt] at (\point) {$\point$};
    }

     \draw (A) -- (B);
     \draw (A) -- (M);
     \draw (B) -- (M);
     \draw (A) -- (P);
     \draw[dashed] (B) -- (P);
     \draw (A) -- (N);
     \draw[dashed] (B) -- (N);
     \draw (N) -- (M);
     \draw (N) -- (P);
     \draw (P) -- (M);
     \draw[dashed] (K') -- (N);
     \draw[dashed] (N) -- (K);
     \draw[dashed] (P) -- (K);
\end{tikzpicture}
     \caption{The fifth point \(N\)}\label{fig:the fifth point}
\end{figure}
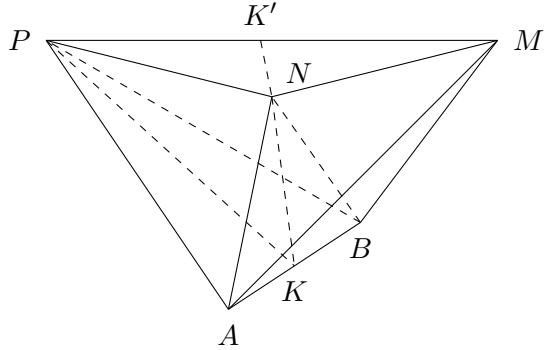

Now choose a point \(N\) on the line segment \(\mathcal{L}\) so that \(|NK|=b'\) and \(N\) lies on the same side of \(K\) as \(K'\), see~\cref{fig:the fifth point}. Writing \(x:=|KK'|\), by right-triangle decomposition with respect to \(\mathcal{L}\) we have
\[
|PK|^2 = x^2+h^2={c'}^2,
\qquad
|PN|^2 = (b'-x)^2+h^2.
\]
In the auxiliary triangle \(\mathcal{T}'\), the foot of the altitude of length \(h\) onto the side \(b'\) splits it into segments of lengths
\[
x=\frac{b'^2+c'^2-a^2}{2b'}\quad\text{and}\quad b'-x,
\]
and \(h^2=c'^2-x^2\). Since \(h=\frac{\ell_{a,b,c}(\varepsilon)}{2}\), we indeed have \(x=\sqrt{{c'}^2-h^2}=|KK'|\). Consequently,
\[
|PN|^2=(b'-x)^2+h^2=a^2.
\]
Finally, because \(N\in \mathcal{L}\) (\(\mathcal{L}\) is the perpendicular bisector of \(AB\)), we have
\[
|NA|^2=|NK|^2+|KA|^2=b'^2+(\varepsilon/2)^2=b^2,
\]
and likewise \(|NB|=b\). This proves the claim.
\end{poc}

Fix the point \(N\) given by \cref{prop:existence5points}. Without loss of generality, assume \(\chi(A)=1\) and \(\chi(B)=2\). Note that, by construction, 
\(\triangle NPA\cong\triangle NPB\cong\triangle NMA\cong\triangle NMB\cong\mathcal{T}\). We distinguish two cases.

\paragraph{Case 1:} \(\chi(P) \notin \{1,2\}\).  
Since neither \(\triangle NPA\) nor \(\triangle NPB\) is rainbow, it follows that \(\chi(N) = \chi(P)\).  
Similarly, the fact that neither \(\triangle NMA\) nor \(\triangle NMB\) is rainbow yields \(\chi(M) = \chi(N) = \chi(P)\).

\paragraph{Case 2:} \(\chi(P) \in \{1,2\}\).  
By symmetry, it suffices to consider \(\chi(P) = 1\) (the case \(\chi(P) = 2\) is analogous).  
Since \(\triangle NPA\) is not monochromatic, we have \(\chi(N) \neq 1\).  
Moreover, \(\triangle NPB\) not being rainbow forces \(\chi(N) = 2\).  
Finally, from \(\triangle NMA\) and \(\triangle NMB\) we deduce \(\chi(M) = 1 = \chi(P)\).

In both cases, \(\chi(M) = \chi(P)\). This finishes the proof.

\begin{figure}[htbp]
\centering
\begin{minipage}{0.45\textwidth}
\centering
\begin{tikzpicture}[scale=1.5]   
    \coordinate (A) at (0,0,1);
    \coordinate (B) at (0.4,0,-1);
    \coordinate (N) at (0,1.5,0);
    \coordinate (P) at (-2,2,0);
    \coordinate (M) at (2,2,0);

    \filldraw[red] (A) circle (0pt) node[below] {$A$};
    \filldraw[blue] (B) circle (0pt) node[below] {$B$};
    \filldraw[green] (P) circle (0pt) node[left] {$P$};
    \filldraw[green] (N) circle (0pt) node[above] {$N$};
    \filldraw[green] (M) circle (0pt) node[right] {$M$};

     \draw (A) -- (B) node[midway,below]{$\varepsilon$} ;
     \draw (A) -- (M) node[midway,below]{$c$};
     \draw (B) -- (M) node[midway,below right]{$c$};
     \draw (A) -- (P) node[midway,below left]{$c$};
     \draw[dashed] (B) -- (P) node[midway,above right]{$c$};
     \draw (A) -- (N) node[midway,above left]{$b$};
     \draw[dashed] (B) -- (N) node[midway,above right]{$b$};
     \draw (N) -- (M) node[midway,above]{$a$};
     \draw (N) -- (P) node[midway,above]{$a$};
     \draw (P) -- (M) node[midway,above]{$\ell_{a,b,c}(\varepsilon)$};

     \fill[brown!30, opacity=0.6] (-2,2,0) -- (0,0,1) -- (0,1.5,0) -- cycle;
     \fill[yellow!30, opacity=0.6] (-2,2,0) -- (0.4,0,-1) -- (0,1.5,0) -- cycle;
\end{tikzpicture}
\caption{Case 1}
\end{minipage}
\hfill
\begin{minipage}{0.45\textwidth}
\centering
\begin{tikzpicture}[scale=1.5]   
    \coordinate (A) at (0,0,1);
    \coordinate (B) at (0.4,0,-1);
    \coordinate (N) at (0,1.5,0);
    \coordinate (P) at (-2,2,0);
    \coordinate (M) at (2,2,0);

    \filldraw[red] (A) circle (0pt) node[below] {$A$};
    \filldraw[blue] (B) circle (0pt) node[below] {$B$};
    \filldraw[red] (P) circle (0pt) node[left] {$P$};
    \filldraw[blue] (N) circle (0pt) node[above] {$N$};
    \filldraw[red] (M) circle (0pt) node[right] {$M$};

     \draw (A) -- (B) node[midway,below]{$\varepsilon$} ;
     \draw (A) -- (M) node[midway,below]{$c$};
     \draw (B) -- (M) node[midway,below right]{$c$};
     \draw (A) -- (P) node[midway,below left]{$c$};
     \draw[dashed] (B) -- (P) node[midway,above right]{$c$};
     \draw (A) -- (N) node[midway,above left]{$b$};
     \draw[dashed] (B) -- (N) node[midway,above right]{$b$};
     \draw (N) -- (M) node[midway,above]{$a$};
     \draw (N) -- (P) node[midway,above]{$a$};
     \draw (P) -- (M) node[midway,above]{$\ell_{a,b,c}(\varepsilon)$};
     \fill[brown!30, opacity=0.6] (2,2,0) -- (0,0,1) -- (0,1.5,0) -- cycle;
     \fill[yellow!30, opacity=0.6] (2,2,0) -- (0.4,0,-1) -- (0,1.5,0) -- cycle;
\end{tikzpicture}
\caption{Case 2}
\end{minipage}
\caption{$\triangle PNA \cong\triangle PNB \cong\triangle MNA \cong\triangle MNB\cong\mathcal{T}$, forcing $\chi(P)=\chi(M)$ whenever $\chi(A)\neq \chi(B)$.}
\label{fig:2Cases}
\end{figure}
\end{proof}

The preceding lemma implies the existence of a relatively large monochromatic sphere.

\begin{lemma}\label{cor:monochromatic sphere}
Let \(r\ge 3\) be an integer. Let \(\chi \colon \mathbb{E}^4 \to [r]\) be an \(r\)-coloring that contains neither a monochromatic nor a rainbow copy of the fixed obtuse triangle \(\mathcal{T}\) with side lengths \(a\le b<c\).
Suppose \(A,B \in \mathbb{E}^4\) satisfy \(\chi(A) \neq \chi(B)\) and \(|AB| = \varepsilon\) with
\[
\ell_{a,b,c}(\varepsilon)\;<\;2\sqrt{\,c^2-(\varepsilon/2)^2\,}.
\]
Let \(\Gamma\) be the \(2\)-sphere centered at the midpoint of \(AB\) consisting of all points \(Z\) with \(|ZA| = |ZB| = c\).
Then \(\Gamma\) is monochromatic under \(\chi\).
\end{lemma}
\begin{proof}[Proof of~\cref{cor:monochromatic sphere}]
The radius of \(\Gamma\) is \(\sqrt{\,c^{2}-(\varepsilon/2)^{2}\,}\). We first need an auxiliary chain lemma on spheres, which is not only useful in \(4\)-dimensional space.

\begin{prop}\label{lem: polygonal chain}
Let \(N\ge 2\) be an integer, \(s>0\), and let \(\Gamma_N(s)\) be an \(N\)-sphere of radius \(s\).
Given distinct \(U,V\in \Gamma_N(s)\) and any \(d\in(0,2s)\), there exist \(k\in\mathbb{N}\) and points
\(X_1,\ldots,X_k\in \Gamma_N(s)\) such that
\[
|UX_1|=|X_1X_2|=\cdots=|X_{k-1}X_k|=|X_kV|=d.
\]
\end{prop}

\begin{proof}[Proof of~\cref{lem: polygonal chain}]
    When \(U,V\) are not antipodal, that is, $|UV|<2s$, we can first pick some $p$ such that $s>p>\max\{\frac{d}{2}, \frac{|UV|}{2}\}$. Choose integer $k$ large enough so that \[
        (k+1)>\frac{|UV|}{d}\sqrt{\frac{1-(d/2p)^2}{1-(|UV|/2p)^2}},
    \]
    and
    \[  
    \left(2(k+1)\arcsin{\frac{d}{2p}}-2\arcsin{\frac{|UV|}{2p}}\right)-\left(2(k+1)\arcsin{\frac{d}{2s}}-2\arcsin{\frac{|UV|}{2s}}\right) > 2\pi.
    \]
   Consider a map $f:[p,s] \to \mathbb{R}$ by \[f(x)=2(k+1)\arcsin{\frac{d}{2x}}-2\arcsin{\frac{|UV|}{2x}}.\]
   Notice that 
    \[\frac{\partial f}{\partial x}=-\frac{1}{x^2}\left((k+1)\frac{d}{\sqrt{1-(d/2x)^2}}-\frac{|UV|}{\sqrt{1-(|UV|/2x)^2}}\right)<0.\] 
    Hence the range of $f$ is $[f(s),f(p)]$ and $f(p)-f(s)>2\pi$. By intermediate value theorem, there exists some $s'\in [p,s]$ such that $f(s')$ is a multiple of $2\pi$.
    
    Fix such \(s'\). Because \(|UV|<2s'\), there is a circle \(\Gamma_{1}(s')\) of radius \(s'\) on \(\Gamma_N(s)\) that contains \(U\) and \(V\) (the intersection with a suitable affine \(2\)-plane). Let \(\mathbb{P}\cong\mathbb{E}^2\) be the plane containing \(\Gamma_{1}(s')\) and its center \(O\). Endow \(\mathbb{P}\) with polar coordinates taking \(O\) as the pole and \(\overrightarrow{OU}\) as polar axis. For each \(i\in[k]\), let \(X_i\) be the point on \(\Gamma_{1}(s')\) with polar angle \(2i\arcsin(\tfrac{d}{2s'})\). Then
    \[
      |UX_1|=|X_1X_2|=\cdots=|X_{k-1}X_k|=|X_kV|=d,
    \]
    where the last equality uses \(f(s')\in 2\pi\mathbb{Z}\).

   When \(U\) and \(V\) are antipodal, first choose \(U'\in \Gamma_N(s)\) with \(|UU'|=d\) and \(|U'V|<2s\), and apply the previous case to the pair \((U',V)\). This completes the proof.

\end{proof}

   Let \(U\) and \(V\) be two arbitrary points on \(\Gamma\). Since \(\Gamma\) is a \(2\)-sphere of radius \(\sqrt{c^2-(\varepsilon/2)^2}\), we apply \cref{lem: polygonal chain} to find a sequence of vertices  \(X_1, X_2,\dots, X_k\) also on the sphere such that 
    \[
    |UX_1|=|X_1X_2|=|X_2X_3|=\cdots=|X_{k-1}X_k|=|X_kV|=\ell_{a,b,c}(\varepsilon).
    \] 
     Notice that for any distinct points \(Z_{1},Z_{2}\) on the sphere \(\Gamma\) with \(|Z_{1}Z_{2}|=\ell_{a,b,c}(\varepsilon)< 2\sqrt{c^2-(\varepsilon/2)^2},\) the tetrahedron \(Z_{1}Z_{2}AB\) is congruent to the tetrahedron \(PMAB\) in \cref{fig:2Cases}. Then by~\cref{lem:5-point structure}, we have \(\chi(Z_{1})=\chi(Z_{2}).\) 
     Therefore, we have \(\chi(U)=\chi(X_1)=\cdots=\chi(X_k)=\chi(V).\) Since \(U,V\) are chosen arbitrarily on \(\Gamma\), the sphere \(\Gamma\) is monochromatic. 
\end{proof}

\subsection{Completion of the proof of Theorem~\ref{thm:triangle}}
Assume for a contradiction that Theorem~\ref{thm:triangle} fails. Then there exists
an \(r\)-coloring \(\chi:\mathbb{E}^4\to[r]\) (with \(r\ge 3\)) that contains neither a monochromatic nor a rainbow copy of the fixed obtuse triangle \(\mathcal{T}\) with side lengths \(a\le b<c\).

Let \(\gamma\) be the largest angle of \(\mathcal{T}\), so that \(\operatorname{rad}(\mathcal{T})=\frac{c}{2\sin\gamma}\).

\smallskip
\noindent\textbf{Case A: \(\gamma<\tfrac{5\pi}{6}\).}
Then \(\sin\gamma>\sin(\tfrac{5\pi}{6})=\tfrac12\), hence \(\operatorname{rad}(\mathcal{T})<c\).
Choose \(\varepsilon\in\bigl(0,h(a,b,c)\bigr)\) such that
\[
(\varepsilon/2)^2 \;<\; c^2-\operatorname{rad}(\mathcal{T})^2 .
\]
By Proposition~\ref{prop:TwoPoint} there exist two points \(A,B\) of different colors with \(|AB|=\varepsilon\).
Applying \cref{cor:monochromatic sphere} to \(A,B\) (note that \(\ell_{a,b,c}(\varepsilon)\) is defined and \(\ell_{a,b,c}(\varepsilon)<2\sqrt{\,c^2-(\varepsilon/2)^2\,}\) by Proposition~\ref{prop:KeyInequality}\textup{(2)}), we obtain a monochromatic sphere
\[
\mathcal{S}=\{Z:\ |ZA|=|ZB|=c\}
\]
centered at the midpoint \(O\) of \(AB\), with \(\operatorname{rad}(\mathcal{S})=\sqrt{\,c^2-(\varepsilon/2)^2\,}\).
By the choice of \(\varepsilon\) we have \(\operatorname{rad}(\mathcal{S})>\operatorname{rad}(\mathcal{T})\).
Since any \(2\)-sphere of radius \(R\) contains circles of every radius \(\rho\in(0,R]\), we may take a circle \(\mathcal{C}\subseteq\mathcal{S}\) of radius \(\operatorname{rad}(\mathcal{T})\) and inscribe on \(\mathcal{C}\) a congruent copy of \(\mathcal{T}\), yielding a monochromatic copy of \(\mathcal{T}\), a contradiction.

\smallskip
\noindent\textbf{Case B: \(\gamma\ge\tfrac{5\pi}{6}\).}
Then \(\operatorname{rad}(\mathcal{T})\ge c\). Pick an arbitrary \(\varepsilon\) with
\[
\varepsilon\in\biggl(0,\min\Bigl\{h(a,b,c),\tfrac{c}{2}\Bigr\}\biggr),
\]
and take \(A,B\) with \(|AB|=\varepsilon\) and \(\chi(A)\ne\chi(B)\) by~\cref{prop:TwoPoint}. 
By \cref{cor:monochromatic sphere} we again obtain a monochromatic (say, red) \(2\)-sphere
\[
\mathcal{S}:=\{Z:\ |ZA|=|ZB|=c\}
\]
centered at the midpoint \(O\) of \(AB\), with \(\operatorname{rad}(\mathcal{S})=\sqrt{\,c^2-(\varepsilon/2)^2\,}<c\).
Let \(\mathbb{H}\cong\mathbb{E}^3\) be the \(3\)-dimensional space spanned by \(\mathcal{S}\).
Note that \(\operatorname{rad}(\mathcal{S})<c\le \operatorname{rad}(\mathcal{T})\).

Fix any chord \(CD\subseteq \mathcal{S}\) with \(|CD|=c\). In the plane \(\operatorname{aff}\{O,C,D\}\)
there exists a point \(Z\) (off \(\mathcal{S}\)) with \(|ZC|=a\), \(|ZD|=b\), and such that the segment \(OZ\) intersects the open segment \(CD\) (this is the standard construction of a triangle with prescribed side lengths on a circle of radius \(\operatorname{rad}(\mathcal{S})\); see~\cref{fig:non-red sphere}. In particular, \(|OZ|\) is uniquely determined by \(a,b,c\) and \(\varepsilon\)). Since \(C,D\) are red and \(\chi\) contains no monochromatic copy of \(\mathcal{T}\), the point \(Z\) cannot be red. By rotating the chord \(CD\) about \(O\) within \(\mathbb{H}\), we see that every point at the same distance \(|OZ|\) from \(O\) is also non-red (indeed, for each such point \(Z'\) one can choose a chord \(C'D'\subseteq \mathcal{S}\) with \(|C'D'|=c\) so that \(\triangle Z'C'D'\cong \mathcal{T}\) in the plane \(\operatorname{aff}\{O,Z'\}\)).
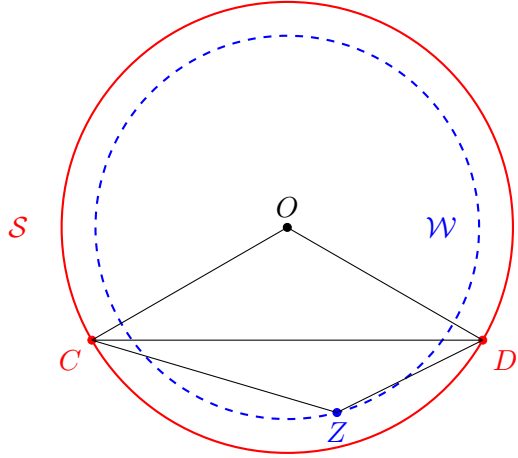
\begin{figure}[H]
\centering
 \begin{tikzpicture}[scale=1.5]
    
    \coordinate (O) at (0,0);
    \filldraw[black] (O) circle (1pt) node[above] {$O$};
    \draw[thick, red] (0,0) circle (2) node[left=33mm] {$\mathcal{S}$};
    
    \coordinate (C) at (210:2);
    \coordinate (D) at (330:2);
    \filldraw[red] (C) circle (1pt) node[below left] {$C$};
    \filldraw[red] (D) circle (1pt) node[below right] {$D$};

    \coordinate (Z) at (285:1.7);
    \filldraw[blue] (Z) circle (1pt) node[below] {$Z$};
    \draw[thick, blue, dashed] (0,0) circle (1.7) node[right=17mm] {$\mathcal{W}$};

    \draw (O) -- (C);
    \draw (O) -- (D);
    \draw (C) -- (Z);
    \draw (D) -- (Z);
    \draw (C) -- (D);
\end{tikzpicture}
\caption{Notice that $\chi(C)=\chi(D)=$red since \(\mathcal{S}\) is red, then $\chi(Z)\neq$ red for $\triangle ZCD \cong \mathcal{T}$, which implies that all points in the sphere $\mathcal{W}$ are not colored red.}
\label{fig:non-red sphere}
\end{figure}
Hence all points on the \(2\)-sphere
\[
\mathcal{W}:=\{X\in\mathbb{H}:\ |OX|=|OZ|\}
\]
are non-red. In particular,
\[
\operatorname{rad}(\mathcal{W})=|OZ|>d(O,CD)
=\sqrt{\,\operatorname{rad}(\mathcal{S})^2-(c/2)^2\,}
=\sqrt{\,\tfrac{3}{4}c^2-\tfrac{\varepsilon^2}{4}\,}.
\]
Indeed, since the segment \(OZ\) meets the chord \(CD\), the point of intersection realizes the minimal distance from \(O\) to \(CD\), namely \(d(O,CD)\), while \(Z\) lies further along the same line.

We thus have a red sphere \(\mathcal{S}\) and a non-red sphere \(\mathcal{W}\). Our next goal is to find another monochromatic sphere \(\mathcal{F}\) that intersects both \(\mathcal{S}\) and \(\mathcal{W}\), leading to a contradiction. Define
\[
\rho\ :=\ \inf\Bigl\{\,r\in[0,\operatorname{rad}(\mathcal{S})] \ \Big|\ \{X\in\mathbb{H}: r\le |OX|\le \operatorname{rad}(\mathcal{S})\}\ \text{is entirely red}\,\Bigr\}.
\]
From the existence of the non-red sphere \(\mathcal{W}\) we have \(\rho\ge \operatorname{rad}(\mathcal{W})>\sqrt{\,\tfrac{3}{4}c^2-\tfrac{\varepsilon^2}{4}}\,\).
Now choose \(\delta>0\) small enough so that
\begin{equation}\label{eq:delta1}
    \delta \;<\; \frac{c^2-\rho^2}{2\rho},
\end{equation}
\begin{equation}\label{eq:delta2}
    \delta \;<\; \frac{\operatorname{rad}(\mathcal{S})^2-\rho^2}{2\rho}
    \qquad\text{if }\ \rho<\operatorname{rad}(\mathcal{S}),
\end{equation}
and
\begin{equation}\label{eq:delta3}
    \delta \;<\; \frac{h(a,b,c)^2}{2\rho}.
\end{equation}

\begin{claim}\label{claim:PQ}
There exist distinct points \(P,Q\in\mathbb{H}\) such that
\[
\chi(P)=\textup{red},\quad \chi(Q)\ne\textup{red},\quad \rho-\delta \le |OQ| \le \rho,\quad OQ\perp QP,\quad |PQ|\le \sqrt{2\rho\delta}.
\]
\end{claim}
\begin{poc}
     By the definition of \(\rho\), the spherical shell \(\{X\in\mathbb{H}: \rho-\delta\le |OX|\le \operatorname{rad}(\mathcal{S})\}\) is not entirely red, so there exists a non-red point \(Q\) with \(\rho-\delta \le |OQ| \le \rho\). Since \(\mathcal{S}\) is red, \(Q\notin \mathcal{S}\) and hence \(|OQ|<\operatorname{rad}(\mathcal{S})\). Fix such a point \(Q\). We consider two cases. 

\paragraph{Case 1.} \(\rho=\operatorname{rad}(\mathcal{S})\). Choose \(P\in\mathcal{S}\) with \(OQ\perp QP\) (the circle \(\mathcal{S}\cap\{x\in\mathbb{H}:(x-Q)\perp OQ\}\) is nonempty). Then
\[
0<|PQ|^2=\operatorname{rad}(\mathcal{S})^2-|OQ|^2\le \rho^2-(\rho-\delta)^2=2\rho\delta-\delta^2\le 2\rho\delta,
\]
and \(\chi(P)=\textup{red}\) because \(P\in\mathcal{S}\).

\paragraph{Case 2.} \(\rho<\operatorname{rad}(\mathcal{S})\). Choose \(P\in\mathbb{H}\) with \(OQ\perp QP\) and \(|QP|=\sqrt{2\rho\delta}\) (again the perpendicular circle through \(Q\) meets the sphere of radius \(\sqrt{2\rho\delta}\) about \(Q\)). Then
\[
|OP|^2=|OQ|^2+|PQ|^2\ge (\rho-\delta)^2+2\rho\delta=\rho^2+\delta^2>\rho^2,
\]
while, using \(|OQ|\le \rho\) and \eqref{eq:delta2},
\[
|OP|^2=|OQ|^2+2\rho\delta\le \rho^2+2\rho\delta<\operatorname{rad}(\mathcal{S})^2.
\]
By the definition of \(\rho\), every point with \(|OX|>\rho\) is red; hence \(\chi(P)=\textup{red}\).

In both cases, the required points \(P,Q\) exist, completing the proof.
\end{poc}

Fix such \(P\) and \(Q\) provided in~\cref{claim:PQ}, then by \eqref{eq:delta1} we have
\begin{equation}{\label{delta1}}
    \rho^2 < c^2-2\rho\delta \le c^2-\rho\delta \le c^2 - (|PQ|/2)^2.
\end{equation}
Moreover, $\varepsilon<\frac{c}{2}$ implies $c^2-\varepsilon^2>\varepsilon^2$, and therefore
\[  
    4\rho(\rho-\delta) = 4\rho^2-4\rho\delta > 6\rho^2-2c^2 > \frac{9}{2}c^2-\frac{3}{2}\varepsilon^2-2c^2 = c^2 +\frac{3}{2}(c^2-\varepsilon^2) > c^2 +\varepsilon^2.
\]
Hence
\begin{equation}\label{Ineq:final}
    \frac{(\varepsilon/2)^2}{\sqrt{c^2-\rho\delta}}\le \frac{(\varepsilon/2)^2+(c/2)^2}{\sqrt{c^2-\rho\delta}} \le \frac{(\varepsilon/2)^2+(c/2)^2}{\rho} \le \rho-\delta.
\end{equation}

Consider the perpendicular-bisector plane
\(\mathbb{Y}\cong\mathbb{E}^{2}\) of the segment \(PQ\) in \(\mathbb{H}\), and set
\[
\mathcal{S}':=\mathcal{S}\cap\mathbb{Y},\quad \mathcal{W}':=\mathcal{W}\cap\mathbb{Y}.
\]
These are circles with the same center \(O'\) (the orthogonal projection of \(O\) onto \(\mathbb{Y}\)) and radii
\[
\textup{rad}(\mathcal{S}')=\sqrt{\,\textup{rad}(\mathcal{S})^2-(|PQ|/2)^2\,},\qquad
\textup{rad}(\mathcal{W}')=\sqrt{\,\textup{rad}(\mathcal{W})^2-(|PQ|/2)^2\,}.
\]

Let $K$ be the midpoint of $PQ$. Notice that to find $Z_1\in \mathcal{S}'$ such that $|Z_1P|=|Z_1Q|=c$ is equivalent to find some $Z_1\in \mathcal{S}'$ such that
\[
|Z_1K|^2=|Z_1P|^2-(|PQ|/2)^2=c^2-(|PQ|/2)^2.
\]
That is, we need to show that the circle centered at $K$ of radius $\sqrt{c^2-(|PQ|/2)^2}$ intersects with $\mathcal{S}'$. It then suffices to show that \[|\textup{rad}(\mathcal{S}')-\sqrt{c^2-(|PQ|/2)^2}| \le |O'K|\le \textup{rad}(\mathcal{S}')+\sqrt{c^2-(|PQ|/2)^2}.\] On one hand, since $OO'QK$ forms a rectangle,  
\[
|O'K|=|OQ|\le \rho \le \sqrt{c^2-(|PQ|/2)^2} \le \textup{rad}(\mathcal{S}')+\sqrt{c^2-(|PQ|/2)^2}.
\]
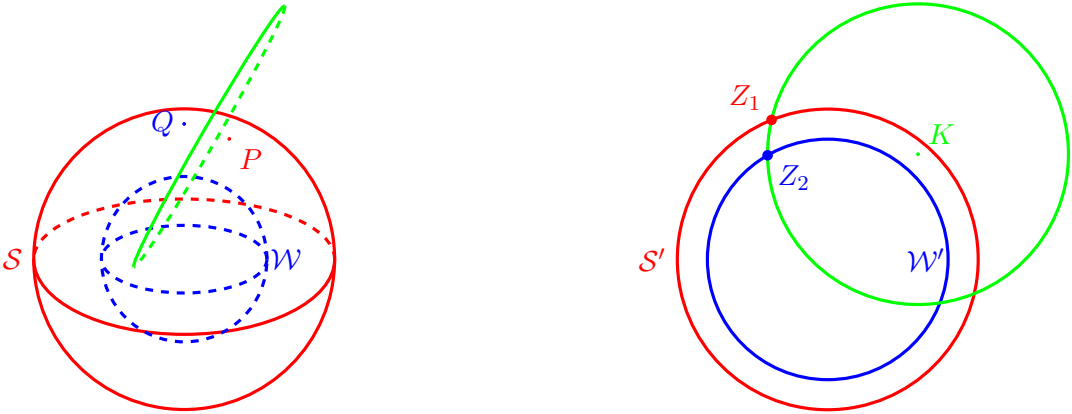
\begin{figure}[htbp]
\centering
\begin{minipage}{0.45\textwidth}
\centering
\begin{tikzpicture}
    \draw[red,very thick,dashed] (0,0) arc (0:180:2cm and 0.8cm);
    \draw[red,very thick] (0,0) arc (0:-180:2cm and 1cm);

    \draw[red,very thick] (-2,0) circle (2cm) node[left=20mm] {$\mathcal{S}$};

    \draw[blue,very thick,dashed] (-0.9,0) arc (0:180:1.1cm and 0.45cm);
    \draw[blue,very thick,dashed] (-0.9,0) arc (0:-180:1.1cm and 0.45cm);

    \draw[blue,very thick,dashed] (-2,0) circle (1.1cm) node[right=10mm] {$\mathcal{W}$};

    \coordinate (Q) at (-2,1.8);
    \filldraw[blue] (Q) circle (0.5pt) node[left] {$Q$};
    \coordinate (P) at (-1.4,1.6);
    \filldraw[red] (P) circle (0.5pt) node[below right] {$P$};
    \begin{scope}[shift={(-0.3,2)}]
    \begin{scope}[rotate=60]
        \draw[green,very thick] (1,1) arc (0:180:2cm and 0.1cm);
        \draw[green,very thick,dashed] (1,1) arc (0:-180:2cm and 0.1cm);
    \end{scope}
    \end{scope}
\end{tikzpicture}
\end{minipage}
\hfill
\begin{minipage}{0.45\textwidth}
\centering
\begin{tikzpicture}
    \draw[name path = circle S,very thick, red] (0,0) circle (2cm) node[left = 20mm] {$\mathcal{S'}$};
    \draw[name path = circle W,very thick, blue] (0,0) circle (1.6cm) node[right = 9mm] {$\mathcal{W'}$};
    \coordinate (Q) at (1.2,1.4);
    \filldraw[green] (Q) circle (0.5pt) node[above right] {$K$};
    \draw[name path = circle P, very thick,green] (1.2,1.4) circle (2cm);
    \path[name intersections={of=circle S and circle P, by={Z1}}];
         \fill[red] (Z1) circle (2pt) node[above left] {$Z_1$};
    \path[name intersections={of=circle W and circle P, by={Z2}}];
         \fill[blue] (Z2) circle (2pt) node[below right] {$Z_2$};
\end{tikzpicture}
\end{minipage}
\caption{The left figure illustrates the monochromatic sphere determined by $P$ and $Q$ intersects both the red and non-red sphere. The right figure shows their projection on $\mathbb{Y}$, which is the perpendicular bisector of segment $PQ$.}
\end{figure}
On the other hand, by~\eqref{Ineq:final} we have
\begin{align*} 
    \left|\textup{rad}(\mathcal{S}')-\sqrt{c^2-(|PQ|/2)^2}\right| &= \left|\sqrt{\textup{rad}(\mathcal{S})^2-(|PQ|/2)^2}-\sqrt{c^2-(|PQ|/2)^2}\right| \\ &= \left| \frac{\textup{rad}(\mathcal{S})^2-c^2}{\sqrt{\textup{rad}(\mathcal{S})^2-(|PQ|/2)^2}+\sqrt{c^2-(|PQ|/2)^2}}\right| \\ &\le \left| \frac{\textup{rad}(\mathcal{S})^2-c^2}{\sqrt{c^2-\rho\delta}}\right| \\ &= \frac{(\varepsilon/2)^2}{\sqrt{c^2-\rho\delta}} \le \rho-\delta.
\end{align*}
Similarly, to find a point $Z_2$ on the sphere \(\mathcal{W}'\) such that \(|Z_{2}P|=|Z_{2}Q|=c\), it suffices to show that 
\[
|\textup{rad}(\mathcal{W}')-\sqrt{c^2-(|PQ|/2)^2}| \le |O'K| \le \textup{rad}(\mathcal{W}')+\sqrt{c^2-(|PQ|/2)^2}.
\] 
We also have 
\[
|O'K|=\rho-\delta\le \sqrt{c^2-(|PQ|/2)^2} \le \textup{rad}(\mathcal{W}')+\sqrt{c^2-(|PQ|/2)^2}
\]
and by~\eqref{Ineq:final} we have
\begin{align*} 
    \left|\textup{rad}(\mathcal{W}')-\sqrt{c^2-(|PQ|/2)^2}\right| &= \left|\sqrt{\textup{rad}(\mathcal{W})^2-(|PQ|/2)^2}-\sqrt{c^2-(|PQ|/2)^2}\right| \\ &= \left| \frac{\textup{rad}(\mathcal{W})^2-c^2}{\sqrt{\textup{rad}(\mathcal{W})^{2}-(|PQ|/2)^2}+\sqrt{c^2-(|PQ|/2)^2}}\right| \\ &\le \left| \frac{\textup{rad}(\mathcal{W})^2-c^2}{\sqrt{c^2-\rho\delta}}\right| \\ &\le  \frac{(\varepsilon/2)^2+(c/2)^2}{\sqrt{c^2-\rho\delta}} \le \rho-\delta.
\end{align*}

Consequently, there exist points \(Z_1\in \mathcal{S}'\) and \(Z_2\in \mathcal{W}'\) with \(|Z_1P|=|Z_1Q|=|Z_2P|=|Z_2Q|=c\).  
By~\eqref{eq:delta3} we have \(|PQ|\le \sqrt{2\rho\delta}< h(a,b,c)\). Then, by Proposition~\ref{prop:KeyInequality}\textup{(2)}, \(\ell_{a,b,c}(|PQ|)<2\sqrt{\,c^2-(|PQ|/2)^2\,}\), so \cref{cor:monochromatic sphere} applies to the pair \((P,Q)\). Hence the \(2\)-sphere \(\{X:|XP|=|XQ|=c\}\) is monochromatic, and in particular \(\chi(Z_1)=\chi(Z_2)\). This contradicts the fact that \(Z_1\in \mathcal{S}'\) is red while \(Z_2\in \mathcal{W}'\) is non-red. This completes the proof.

\section{Rectangles: Proof of~\cref{thm:rectangle}}
We begin by introducing the notation used throughout this proof.

\begin{itemize}
    \item \( \mathcal{S}_n(x) \): the point set of a regular \((n{-}1)\)-dimensional simplex of side length \(x\); that is, a set of \(n\) points \(\boldsymbol{x}_1, \dots, \boldsymbol{x}_n\) such that the Euclidean distance \(\|\boldsymbol{x}_i - \boldsymbol{x}_j\| = x\) for all \(i \neq j\).

    \item \( A \times B \): the Cartesian product of sets \(A\) and \(B\), defined as \(\{(\boldsymbol{a}, \boldsymbol{b}) : \boldsymbol{a} \in A,\, \boldsymbol{b} \in B\}\).

    \item \( \ell_x \): a two-point configuration consisting of a line segment of length \(x\).

    \item \( \mathcal{B}_t(x, y) \): a planar configuration of \(t + 1\) points \(\boldsymbol{v}_1, \boldsymbol{v}_2, \ldots, \boldsymbol{v}_{t+1}\) lying in a common plane, such that
    \[
        \|\boldsymbol{v}_i - \boldsymbol{v}_{i+1}\| = y \quad \text{for all } 1 \le i \le t, \quad \text{and} \quad \|\boldsymbol{v}_1 - \boldsymbol{v}_{t+1}\| = x.
    \]
    That is, a polygonal path of \(a\) consecutive edges of length \(y\), whose endpoints are distance \(x\) apart. Here, it requires that \(t\ge\max\{2,\ceil{\frac{x}{y}}\}\). In particular, \(B_2(x, y)\) forms a triangle.

    \item For \(\mathcal{X} \subseteq \mathcal{Y}\), we write \(\mathcal{Y} \overset{r}{\rightarrow} \mathcal{X}\) if every \(r\)-coloring of \(\mathcal{Y}\) contains a monochromatic configuration congruent to \(\mathcal{X}\). In particular, a configuration \(\mathcal{X} \subseteq \mathbb{E}^n\) is called \emph{Ramsey} if for all \(r\), \(\mathbb{E}^n \overset{r}{\rightarrow} \mathcal{X}\) holds for some \(n = n(\mathcal{X}, r)\).
    \item For \(\mathcal{K}_{1},\mathcal{K}_{2} \subseteq \mathcal{Y}\), we write \(\mathcal{Y} \overset{r}{\rightarrow} (\mathcal{K}_1; \mathcal{K}_2)_{\mathrm{GR}}\) if every \(r\)-coloring of \(\mathcal{Y}\) contains a monochromatic copy of \(\mathcal{K}_1\) or a rainbow congruent copy of \(\mathcal{K}_2\).
\end{itemize}

We will use the following classical results.

\begin{theorem}[\cite{1990JAMS}]\label{thm:RodlFrankl}
Every simplex is Ramsey.
\end{theorem}

\begin{theorem}[\cite{1973JCTA}]\label{EGM}
If \(\mathcal{X}\) and \(\mathcal{Y}\) are Ramsey, then so is \(\mathcal{X}\times \mathcal{Y}\).
\end{theorem}

Next we provide the formal proof of~\cref{thm:rectangle}.

\begin{proof}[Proof of~\cref{thm:rectangle}]
Let \(\mathcal{R}\) be a rectangle with side length \(x\) and \(y\), where \(x> y\). Let \(m := \left\lceil \frac{x}{y} \right\rceil\). By definitions, both \(\mathcal{S}_7(y)\) and \(\mathcal{B}_2(y, x)\) are simplices, and hence are Ramsey by Theorem~\ref{thm:RodlFrankl}. By Theorem~\ref{EGM}, their Cartesian product \(\mathcal{S}_7(y) \times \mathcal{B}_2(y, x)\) is also Ramsey. Therefore, for any integer \(q\), there exists a dimension \(n_0 = n_0(x, y, q)\) such that
\[
\mathbb{E}^{n_0} \overset{q}{\rightarrow} \mathcal{S}_7(y) \times \mathcal{B}_2(y, x).
\]
Set \(q = (m+1)\binom{3m+1}{2} + 3m + 1\), and fix such an \(n_0\). Define \(n := n_0 + 3m + 2\). Note that \(n\) depends only on the rectangle \(\mathcal{R}\), and is independent of the number of colors \(r\). Let \(\chi \colon \mathbb{E}^{n} \to [r]\) be an arbitrary \(r\)-coloring. Observe that when \(r \le 3\), the statement follows trivially since every rectangle is Ramsey~\cite{1973JCTA}. Therefore, we may assume \(r \ge 4\) throughout the proof. We begin with the following key auxiliary lemma.

\begin{lemma}\label{lemma:aux1}
Let \(\mathcal{R}\) be a rectangle with side lengths \(a\) and \(b\). Let \( r \ge 4 \) and \(s\ge \max \{2,\ceil{\frac{a}{b}}\}\) be integers. Then
\[
\mathcal{S}_{3s+1}(a) \times \mathcal{B}_{s}(a, b) \overset{r}{\rightarrow} (\ell_{a}; \mathcal{R})_{\textup{GR}}.
\]
\end{lemma}
\begin{proof}[Proof of Lemma~\ref{lemma:aux1}]
Let \( \tau \colon \mathcal{S}_{3s+1}(a) \times \mathcal{B}_s(a, b) \to [r] \) be an arbitrary \(r\)-coloring of the product configuration. Suppose that there is no monochromatic copy of \(\ell_a\). We will show the existence of a rainbow copy of \(\mathcal{T}\).

Recall that \( \mathcal{B}_s(a, b) \) consists of \(s + 1\) points \(\boldsymbol{v}_1, \boldsymbol{v}_2, \ldots, \boldsymbol{v}_{s+1}\) lying in a common plane, such that
    \[
        \|\boldsymbol{v}_i - \boldsymbol{v}_{i+1}\| = b \quad \text{for all } 1 \le i \le s, \quad \text{and} \quad \|\boldsymbol{v}_1 - \boldsymbol{v}_{s+1}\| = a.
    \]
Fix any point \(\boldsymbol{x}_k \in \mathcal{S}_{3s+1}(x)\), and consider the sequence
\[
(\boldsymbol{x}_k, \boldsymbol{v}_1),\ (\boldsymbol{x}_k, \boldsymbol{v}_2),\ \ldots,\ (\boldsymbol{x}_k, \boldsymbol{v}_{s+1})\in \mathcal{S}_{3s+1}(a)\times \mathcal{B}_{s}(a,b).
\]
By definitions, \(\|(\boldsymbol{x}_k, \boldsymbol{v}_1) - (\boldsymbol{x}_k, \boldsymbol{v}_{s+1})\| = a\). Since \(\tau\) avoids monochromatic copies of \(\ell_a\), these two points must have different colors:
\[
\tau(\boldsymbol{x}_k, \boldsymbol{v}_1) \ne \tau(\boldsymbol{x}_k, \boldsymbol{v}_{s+1}).
\]
Thus, there exists some index \(i \in \{1, \dots, s\}\) such that
\[
\tau(\boldsymbol{x}_k, \boldsymbol{v}_i) \ne \tau(\boldsymbol{x}_k, \boldsymbol{v}_{i+1}).
\]

Now, vary \(\boldsymbol{x}_k\) over all \(3s+1\) points in \(\mathcal{S}_{3s+1}(a)\). For each \(k\in [3s+1]\), associate to it an index \(i_k\) where \(\tau(\boldsymbol{x}_k, \boldsymbol{v}_{i_k}) \ne \tau(\boldsymbol{x}_k, \boldsymbol{v}_{i_k+1})\). By the pigeonhole principle, there exists some fixed \(i\) such that
\[
\tau(\boldsymbol{x}_k, \boldsymbol{v}_i) \ne \tau(\boldsymbol{x}_k, \boldsymbol{v}_{i+1})
\]
holds for at least \(\ceil{\frac{3s+1}{s}}=4\) distinct values of \(k\). Without loss of generality, assume this holds for \(k = 1, 2, 3, 4\). We now examine the colors of the points \((\boldsymbol{x}_k, \boldsymbol{v}_i)\) and \((\boldsymbol{x}_k, \boldsymbol{v}_{i+1})\) for \(k = 1, 2, 3, 4\). Without loss of generality, we may assume
\[
\tau(\boldsymbol{x}_1, \boldsymbol{v}_i) = \text{red}, \quad \tau(\boldsymbol{x}_1, \boldsymbol{v}_{i+1}) = \text{blue}.
\]
For each fixed \(j \in \{i, i+1\}\), the set \(\{(\boldsymbol{x}_k, \boldsymbol{v}_j) : k = 1, 2, 3, 4\}\) forms a regular simplex with all pairwise distances equal to \(a\). Since \(\tau\) avoids monochromatic copies of \(\ell_a\), these four points must all receive distinct colors. In particular, the followings hold.
\begin{itemize}
    \item Among \(\tau(\boldsymbol{x}_k, \boldsymbol{v}_{i+1})\) for \(k = 1, 2, 3, 4\), at most one can be colored red.
    \item Among \(\tau(\boldsymbol{x}_k, \boldsymbol{v}_i)\) for \(k = 1, 2, 3, 4\), at most one can be colored blue.
\end{itemize}
Therefore, there exists some \(k \ne 1\) (without loss of generality, assume \(k = 2\)) such that
\[
\tau(\boldsymbol{x}_2, \boldsymbol{v}_i) \ne \text{blue}, \quad \tau(\boldsymbol{x}_2, \boldsymbol{v}_{i+1}) \ne \text{red}.
\]
Thus, the four vertices
\[
(\boldsymbol{x}_1, \boldsymbol{v}_i), \quad (\boldsymbol{x}_1, \boldsymbol{v}_{i+1}), \quad (\boldsymbol{x}_2, \boldsymbol{v}_i), \quad (\boldsymbol{x}_2, \boldsymbol{v}_{i+1})
\]
receive four distinct colors under \(\tau\). These form a rainbow copy of the rectangle \(\mathcal{R}\) with side lengths \(a\) and \(b\), a contradiction. This completes the proof of~\cref{lemma:aux1}.
\end{proof}

 For each point \(\boldsymbol{u} \in \mathbb{E}^{n_0}\), consider the product configuration
\[
(\boldsymbol{u}, \mathcal{S}_{3m+1}(x) \times \mathcal{B}_m(x, y)) := \{(\boldsymbol{u}, \boldsymbol{v}) : \boldsymbol{v} \in \mathcal{S}_{3m+1}(x) \times \mathcal{B}_m(x, y)\} \subseteq \mathbb{E}^{n_0 + 3m + 2}.
\]
By Lemma~\ref{lemma:aux1}, under the coloring \(\chi\), each such set either contains a rainbow copy of \(\mathcal{R}\), or a monochromatic copy of \(\ell_x\). If a rainbow copy of \(\mathcal{R}\) occurs for some \(\boldsymbol{u}\), we are done. Thus, we may assume that for every \(\boldsymbol{u} \in \mathbb{E}^{n_0}\), no rainbow copy of \(\mathcal{R}\) exists. Then Lemma~\ref{lemma:aux1} ensures that each associated product configuration must contain a monochromatic copy of \(\ell_x\). Since \(\mathcal{S}_{3m+1}(x) \times \mathcal{B}_m(x, y)\) is finite, there are only finitely many possible positions for such \(\ell_x\)'s. More precisely, each copy of \(\ell_x\) arises in one of the following two forms:
\begin{itemize}
    \item A line between two distinct points in \(\mathcal{S}_{3m+1}(x)\), paired with a fixed point in \(\mathcal{B}_m(x,y)\). There are \(\binom{3m+1}{2}\) choices of lines in \(\mathcal{S}_{3m+1}(x)\), and for each such line, \(m+1\) possible choices of a point in \(\mathcal{B}_m(x,y)\), yielding a total of \((m+1) \cdot \binom{3m+1}{2} \) copies.
    
    \item There exists one line in \(\mathcal{B}_m(x,y)\) of length exactly \(x\), namely the line between \(v_1\) and \(v_{m+1}\), and pair it with a fixed point in \(\mathcal{S}_{3m+1}(x)\). Since there are \(3m+1\) choices of points in \(\mathcal{S}_{3m+1}(x)\), this contributes \(3m+1\) copies.
\end{itemize}
Therefore, the number of distinct copies of \(\ell_x\) in \(\mathcal{S}_{3m+1}(x) \times \mathcal{B}_m(x, y)\) is bounded by
\[
q= (m+1)\binom{3m+1}{2} + 3m+1.
\]
To formalize the selection of these monochromatic \(\ell_x\)'s, we define a labeling function
\[
\lambda : \{ \ell_x \subseteq \mathcal{S}_{3m+1}(x) \times \mathcal{B}_m(x,y) \} \to [q],
\]
which assigns a unique label to each possible position of \(\ell_x\).

We define an auxiliary coloring \(\gamma : \mathbb{E}^{n_0} \to [q]\) by the rule:
for each \(\boldsymbol{u} \in \mathbb{E}^{n_0}\), select one monochromatic copy of \(\ell_x\), denoted by \(\ell(\boldsymbol{u})\), within its associated product configuration under \(\chi\), and set
\[
\gamma(\boldsymbol{u}) := \lambda( \ell(\boldsymbol{u}) ).
\]
That is, \(\gamma(\boldsymbol{u})\) records the position label of the chosen monochromatic \(\ell_x\). By definitions of \(n_{0}\) and \(q\),
\[
\mathbb{E}^{n_0} \overset{q}{\rightarrow} \mathcal{S}_7(y) \times \mathcal{B}_2(y, x).
\]
Thus, there exists a monochromatic copy of \(\mathcal{S}_7(y) \times \mathcal{B}_2(y, x)\) under \(\gamma\). Fix this monochromatic copy and denote the set of its points as \(V \subseteq \mathbb{E}^{n_0}\).

By the definition of \(\gamma\), every point \(\boldsymbol{u} \in V\) is associated with the same monochromatic copy of \(\ell_x\), whose endpoints \(\boldsymbol{v}_1, \boldsymbol{v}_2\) lie in \(\mathcal{S}_{3m+1}(x) \times \mathcal{B}_m(x,y)\). That is, for all \(\boldsymbol{u} \in V\), the pair \(\{ (\boldsymbol{u}, \boldsymbol{v}_1), (\boldsymbol{u}, \boldsymbol{v}_2) \}\) forms a monochromatic copy of \(\ell_x\) under \(\chi\). Consider now the set of points
\[
V \times \{ \boldsymbol{v}_1, \boldsymbol{v}_2 \} \subseteq \mathbb{E}^{n_0 + 3m + 2}.
\]
If this set contains a rainbow copy of \(R\) under \( \chi \), we are done. Otherwise, focus on the subset
\[
(V, \boldsymbol{v}_1) := \{ (\boldsymbol{u}, \boldsymbol{v}_1) : \boldsymbol{u} \in V \}.
\]
Since \(V\) is isomorphic to \(\mathcal{S}_7(y) \times \mathcal{B}_2(y, x)\), we can apply Lemma~\ref{lemma:aux1} again with parameter \(s = 2\). Here, we interpret the rectangle \(\mathcal{T}\) as being rotated by \(\frac{\pi}{2}\), effectively treating the side of length \(y\) as the dominant side in the product structure. This flexible application of Lemma~\ref{lemma:aux1} allows us to leverage the same combinatorial argument with the roles of \(x\) and \(y\) interchanged. More precisely, if a rainbow copy of \(\mathcal{T}\) under \(\chi\) is found, we are done; otherwise, Lemma~\ref{lemma:aux1} guarantees the existence of a monochromatic copy of \(\ell_y\), say between the points \((\boldsymbol{u}_1, \boldsymbol{v}_1)\) and \((\boldsymbol{u}_2, \boldsymbol{v}_1)\).

Finally, we make the following observations under the coloring function \(\chi\):
\begin{itemize}
    \item The points \((\boldsymbol{u}_1, \boldsymbol{v}_1)\) and \((\boldsymbol{u}_2, \boldsymbol{v}_1)\) share the same color, as they form a monochromatic copy of \(\ell_y\) constructed in the previous step. 
    \item The points \((\boldsymbol{u}_1, \boldsymbol{v}_1)\) and \((\boldsymbol{u}_1, \boldsymbol{v}_2)\) also share the same color, since the line with endpoints \(\boldsymbol{v}_1, \boldsymbol{v}_2\) was fixed as a monochromatic \(\ell_x\) across all associated product configurations. Likewise, the points \((\boldsymbol{u}_2, \boldsymbol{v}_1)\) and \((\boldsymbol{u}_2, \boldsymbol{v}_2)\) are colored identically.
\end{itemize}
Therefore, the four points \( (\boldsymbol{u}_1, \boldsymbol{v}_1),  (\boldsymbol{u}_1, \boldsymbol{v}_2), (\boldsymbol{u}_2, \boldsymbol{v}_1), (\boldsymbol{u}_2, \boldsymbol{v}_2)\) form a monochromatic rectangle congruent to \(\mathcal{R}\) under the original coloring \(\chi\). This completes the proof.
\end{proof}

\section{A general framework for low-dimensional simplices}
\subsection{Notation}
Most notation in this section is identical to those above. For the sake of convenience, we list them here again.
\begin{itemize}
    \item We use blackboard bold letters for spaces, in particular, $\mathbb{E}^n$ for the $n$-dimensional Euclidean space. We use calligraphy, uppercase letters such as $\mathcal{X},\mathcal{Y}$ for set of points in $\mathbb{E}^n$.
    \item We use boldface, lowercase letters such as $\bm{a},\bm{b}$ for points in $\mathbb{E}^n$ only; and their difference $\bm{a}-\bm{b}$ for the vector between. Besides, $\|\bm{a}-\bm{b}\|$ represents the usual Euclidean norm of this vector, or equivalently, the distance between points $\bm{a}$ and $\bm{b}.$ Sometimes, we write $\text{dist}(\bm{a},\bm{b})$ for $\|\bm{a}-\bm{b}\|$ if $\bm{a}$ and $\bm{b}$ are expressed in their coordinates.
    \item For three points $\bm{a},\bm{b},\bm{c}$, we use $\angle \bm{a}\bm{b}{\bm{c}}$ for the induced angle between vectors $\bm{b}-\bm{a}$ and $\bm{b}-\bm{c}.$ 
    Besides, we use $(\bm{b}-\bm{a})\cdot(\bm{b}-\bm{c})$ for the dot product of these two vectors.
    
\end{itemize}

\subsection{Canonical perturbation via the simplex super-Ramsey theorem}
In this part we invoke the \emph{simplex super-Ramsey} theorem, for completeness we record it below.
\begin{theorem}[\cite{1990JAMS}]\label{thm:simplex super ramsey}
For any simplex $\Delta$, there exist positive constants $c$, $\delta$ and $N$ only depend on $\Delta$, and a configuration $\mathcal{X}\subseteq \mathbb{E}^{n}$ for every $n>N$ with the following two properties: 
    \begin{itemize}
        \item $|\mathcal{X}|<c^n$;
       \item $|\mathcal{Y}| < |\mathcal{X}| \cdot (1+\delta)^{-n}$ holds for all subsets $\mathcal{Y} \subseteq \mathcal{X}$ that contain no congruent copy of $\Delta$.
    \end{itemize}
   
\end{theorem}
For convenience, we restate~\cref{thm:canonical-perturbation} here.

\begin{theorem}\label{thm:canonical perturbation}
Let \(d,n_0,m,r\ge 1\) be integers and let \(\Delta\subseteq\mathbb{E}^{d}\) be a simplex.
Let \(\mathcal{X}=\{\boldsymbol{v}_1,\dots,\boldsymbol{v}_m\}\subseteq \mathbb{P}_0\cong\mathbb{E}^{n_0}\) be any \(m\)-point set, and fix \(\varepsilon>0\).
Then there exist integers \(n_1,\dots,n_m\) with the following property: for every \(r\)-coloring \(\chi\) of the orthogonal product
\[
\mathbb{P}:=\mathbb{P}_0\times\mathbb{P}_1\times\cdots\times\mathbb{P}_m,
\qquad \mathbb{P}_i\cong\mathbb{E}^{n_i}\ (i=0,1,\dots,m),
\]
that contains no monochromatic copy of \(\Delta\), there exist points \(\boldsymbol{x}_i\in\mathbb{P}_i\) for \(i=1,\dots,m\) such that, writing
\(\boldsymbol{p}_i:=(\boldsymbol{v}_i,\boldsymbol{x}_1,\dots,\boldsymbol{x}_m)\in\mathbb{P}\),
each \(\varepsilon\)-sphere \(S_i(\boldsymbol{p}_i;\varepsilon)\) is not monochromatic under \(\chi\).
\end{theorem}

\begin{proof}[Proof of~\cref{thm:canonical perturbation}]
   Let \(\Delta\) consist of \(d+1\) points. Choose a congruent copy
\(\{\boldsymbol{w}_{0},\boldsymbol{w}_{1},\ldots,\boldsymbol{w}_{d}\}\subseteq\mathbb{E}^{d}\) of \(\Delta\) with
\(\boldsymbol{w}_{0}=\boldsymbol{0}\).
Let \(A_{1}\) be the \((d+1)\times d\) matrix whose rows are
\(\boldsymbol{w}_{0},\ldots,\boldsymbol{w}_{d}\).
For each \(i\in[d]\), pick an integer \(m_{i}\ge 2\) and set
\(\varepsilon_{i}:=\|\boldsymbol{w}_{i}\|/m_{i}<\varepsilon\).
Define
\[
n_{1}:=\sum_{i=1}^{d}(m_{i}-1)+d .
\]
Let $A_2$ be an \(n_{1}-d\) by \(d\) matrix with rows of the form \[ \left\{\frac{j}{m_{i}}\cdot \boldsymbol{w}_{i} \biggm| 1\le i\le d, 1\le j \le m_i-1\right\}. \]
Form the block matrix \(A\in \mathbb{R}^{(n_{1}+1)\times n_{1}}\) by
\[
A \;:=\;
\begin{bmatrix}
A_{1} & 0\\[2pt]
A_{2} & (\varepsilon/2)\,I_{n_{1}-d}
\end{bmatrix},
\]
where \(I_{n_{1}-d}\) is the identity matrix of size \(n_{1}-d\).
Let \(\mathcal{B}\subseteq\mathbb{P}_{1}\cong\mathbb{E}^{n_{1}}\) be the set of row vectors of \(A\),
and enumerate it as \(\mathcal{B}=\{\boldsymbol{b}_{1},\ldots,\boldsymbol{b}_{n_{1}+1}\}\).
Since the lower-right block is \((\varepsilon/2)I_{n_{1}-d}\), the rows of \(A\) are
affinely independent; in particular, \(\mathcal{B}\) is a simplex in \(\mathbb{E}^{n_{1}}\).

Apply Theorem~\ref{thm:simplex super ramsey} to the simplex \(\mathcal{B}\), there exist positive constants \(c(\mathcal{B})\), \(\delta(\mathcal{B})\), and \(N(\mathcal{B})\)
(depending only on \(\mathcal{B}\))
such that, for each \(i=2,\ldots,m\), we can choose an integer \(n_{i}\) with
\[
n_{i} \ge \max\Big\{\,N(\mathcal{B}),\ \log_{1+\delta(\mathcal{B})}\!\Bigl(
\,|\mathcal{B}|\cdot\prod_{j=2}^{i-1} c(\mathcal{B})^{\,n_{j}} \Bigr)\Bigr\},
\]
and a set \(\mathcal{X}_{i}\subseteq\mathbb{E}^{n_{i}}\) such that
\begin{itemize}
\item \(|\mathcal{X}_{i}|< c(\mathcal{B})^{\,n_{i}}\);
\item for every \(\mathcal{Y}\subseteq \mathcal{X}_{i}\) containing no congruent copy of \(\mathcal{B}\),
we have \(|\mathcal{Y}|< |\mathcal{X}_{i}|\cdot(1+\delta(\mathcal{B}))^{-n_{i}}\).
\end{itemize}
We now fix the integers \(n_{1},\ldots,n_{m}\) and the sets \(\mathcal{X}_{2},\ldots,\mathcal{X}_{m}\) above.
In what follows we work inside the finite grid
\[
\mathcal{G}\ :=\ \mathcal{B}\times \mathcal{X}_{2}\times\cdots\times \mathcal{X}_{m}
\ \subseteq\ \mathbb{P}_{1}\times\cdots\times\mathbb{P}_{m},
\]
and pair this tail product with the anchors \(\boldsymbol{v}_{1},\ldots,\boldsymbol{v}_{m}\in \mathbb{P}_{0}\cong\mathbb{E}^{n_{0}}\) as needed.

\begin{claim}\label{claim:NotMono}
Fix \(h\in[m]\) and a tuple of fiber points
\(\boldsymbol{x}_{1},\ldots,\boldsymbol{x}_{h-1},\boldsymbol{x}_{h+1},\ldots,\boldsymbol{x}_{m}\),
where \(\boldsymbol{x}_{i}\in\mathbb{P}_{i}\cong\mathbb{E}^{n_{i}}\).
Identify \(\mathcal{B}\) with a fixed isometric copy \(\mathcal{B}^{(h)}\subseteq\mathbb{P}_{h}\)
(of the simplex constructed above) and, for brevity, still write \(\mathcal{B}\).
Then there exists \(\boldsymbol{b}\in \mathcal{B}\) such that the \(\varepsilon\)-sphere
\(S_{h}\bigl((\boldsymbol{v}_{h}, \boldsymbol{x}_{1}, \ldots,\boldsymbol{x}_{h-1},\boldsymbol{b},\boldsymbol{x}_{h+1},\ldots,\boldsymbol{x}_{m});\varepsilon\bigr)\)
in \(\mathbb{P}_{h}\) is not monochromatic under \(\chi\).
\end{claim}
\begin{poc}
Fix \(h\) and the tuple \(\boldsymbol{x}_{i}\) for \(i\neq h\) as above, and abbreviate
\[
C(\boldsymbol{b})\ :=\ (\boldsymbol{v}_{h}, \boldsymbol{x}_{1}, \ldots,\boldsymbol{x}_{h-1},\boldsymbol{b},\boldsymbol{x}_{h+1},\ldots,\boldsymbol{x}_{m})
\ \in\ \mathbb{P}.
\]
Consider the family of \(\varepsilon\)-spheres in the \(h\)-th fiber (Definition~\ref{def:eps-sphere})
\[
\mathcal{S}:=\ \Bigl\{\, S_{h}\bigl(C(\boldsymbol{b});\varepsilon\bigr) : \boldsymbol{b}\in\mathcal{B}\Bigr\}
\subseteq\mathbb{P}_{h}.
\]

\paragraph{Connectivity.}
We first claim that the union \(\bigcup_{S\in\mathcal{S}} S\) is connected in \(\mathbb{P}_{h}\).
Indeed, recall that \(\mathcal{B}\) consists of the rows of the block matrix
\(\begin{bmatrix}A_{1}&0\\ A_{2}&(\varepsilon/2)I\end{bmatrix}\).
Thus \(\mathcal{B}\) contains the \(d+1\) \emph{base points} \((\boldsymbol{w}_{i},\boldsymbol{0})\) with \(0\le i\le d,\)
and for each \(i\in[d]\) and \(j=1,\dots,m_{i}-1\), points of the form
\(\bigl(\tfrac{j}{m_{i}}\cdot\boldsymbol{w}_{i},\,\frac{\varepsilon}{2}\cdot\boldsymbol{e}_{k(i,j)}\bigr)\),
where \(\boldsymbol{e}_{k}\) denotes a standard basis vector in the lower block
(the precise index \(k(i,j)\) is irrelevant).
Two centers in \(\mathcal{B}\) that differ by increasing or decreasing a single step along the
\((i,j)\)-chain have distance
\[
\Bigl\|\,\tfrac{1}{m_{i}}\cdot\boldsymbol{w}_{i}\Bigr\|^{2}+\Bigl\|\frac{\varepsilon}{2}\cdot\boldsymbol{e}_{k}- \frac{\varepsilon}{2}\cdot\boldsymbol{e}_{k'}\Bigr\|^{2}
 \le \varepsilon_{i}^{2}+\Bigl(\tfrac{\varepsilon}{2}\Bigr)^{2}+\Bigl(\tfrac{\varepsilon}{2}\Bigr)^{2}
 < (2\varepsilon)^{2},
\]
because \(\varepsilon_{i}=\|\boldsymbol{w}_{i}\|/m_{i}<\varepsilon\).
Hence successive centers along each \((i,j)\)-chain are at distance \(< 2\varepsilon\).
Moreover, the endpoints of each chain connect to the base points:
\[
\operatorname{dist}\Bigl(\Bigl(\frac{1}{m_{i}}\cdot\boldsymbol{w}_{i},\frac{\varepsilon}{2}\cdot\boldsymbol{e}_{k}\Bigr),(\boldsymbol{0},\boldsymbol{0})\Bigr)
=\sqrt{\varepsilon_{i}^{2}+(\varepsilon/2)^{2}}<2\varepsilon,
\]
\[
\operatorname{dist}\Bigl(\Bigl(\frac{m_{i}-1}{m_{i}}\cdot\boldsymbol{w}_{i},\frac{\varepsilon}{2}\cdot\boldsymbol{e}_{k'}\Bigr),(\boldsymbol{w}_{i},\boldsymbol{0})\Bigr)
=\sqrt{\varepsilon_{i}^{2}+(\varepsilon/2)^{2}}<2\varepsilon.
\]
Thus the graph on \(\mathcal{B}\) that joins two centers when their distance is \(\le 2\varepsilon\)
is connected; since two radius-\(\varepsilon\) spheres intersect whenever their centers are
at distance \(\le 2\varepsilon\), the union \(\bigcup_{S\in\mathcal{S}} S\) is connected.

\paragraph{Propagation of color.}
Assume for contradiction that every sphere in \(\mathcal{S}\) is monochromatic under \(\chi\).
Whenever two monochromatic spheres intersect, they must have the same color on the
intersection, hence the same color on the whole spheres.
By connectivity, all spheres in \(\mathcal{S}\) have a common color.

In particular, for the \(d+1\) base centers \((\boldsymbol{w}_{j},\boldsymbol{0})\in\mathcal{B}\)
(\(0\le j\le d\)), the spheres
\(S_{h}\bigl(C((\boldsymbol{w}_{j},\boldsymbol{0}));\varepsilon\bigr)\) all have this color.
Fix any unit vector \(\boldsymbol{e}_{h}\in\mathbb{P}_{h}\).
Then, for each \(j\), the point
\[
\boldsymbol{q}_{j}\ :=\ \bigl(\boldsymbol{v}_{h}, \boldsymbol{x}_{1}, \ldots,\boldsymbol{x}_{h-1},
(\boldsymbol{w}_{j},\boldsymbol{0})+\varepsilon\boldsymbol{e}_{h},
\boldsymbol{x}_{h+1},\ldots,\boldsymbol{x}_{m}\bigr)
\]
lies on \(S_{h}\bigl(C((\boldsymbol{w}_{j},\boldsymbol{0}));\varepsilon\bigr)\) and hence has that color.
For any \(j,k\) we have
\[
\bigl\|\boldsymbol{q}_{j}-\boldsymbol{q}_{k}\bigr\|
=\bigl\|(\boldsymbol{w}_{j}-\boldsymbol{w}_{k},\,\boldsymbol{0})\bigr\|
=\|\boldsymbol{w}_{j}-\boldsymbol{w}_{k}\|,
\]
so \(\{\boldsymbol{q}_{0},\ldots,\boldsymbol{q}_{d}\}\) is congruent to \(\Delta\).
Thus we obtain a monochromatic copy of \(\Delta\) in \(\mathbb{P}\), contradicting the hypothesis.

Therefore, at least one sphere in \(\mathcal{S}\) is not monochromatic, that is, there exists \(\boldsymbol{b}\in\mathcal{B}\) with the stated property. This finishes the proof.
\end{poc}

\begin{claim}\label{claim:Iteratively}
    For any integer \(1\le i\le m\) and any fixed tail
\((\boldsymbol{x}_{i+1},\ldots,\boldsymbol{x}_{m}) \in \mathcal{X}_{i+1}\times\cdots\times \mathcal{X}_{m}\),
there exists a tuple \((\boldsymbol{x}_{1},\ldots,\boldsymbol{x}_{i}) \in \mathcal{B}\times \mathcal{X}_{2}\times\cdots\times \mathcal{X}_{i}\)
such that, for every \(1\le j\le i\), the \(\varepsilon\)-sphere
\(S_{j}\bigl((\boldsymbol{v}_{j},\boldsymbol{x}_{1},\ldots,\boldsymbol{x}_{m});\varepsilon\bigr)\) in \(\mathbb{P}_{j}\)
is not monochromatic under \(\chi\).
\end{claim}  
\begin{poc}
We proceed by induction on \(i\in[m]\). Let \(P(i)\) denote the statement of the claim.

\paragraph{Base case \(P(1)\).}
Fix any \((\boldsymbol{x}_{2},\ldots,\boldsymbol{x}_{m})\).
By Claim~\ref{claim:NotMono} with \(h=1\), there exists \(\boldsymbol{x}_{1}\in\mathcal{B}\) such that
\(S_{1}\bigl((\boldsymbol{v}_{1},\boldsymbol{x}_{1},\boldsymbol{x}_{2},\ldots,\boldsymbol{x}_{m});\varepsilon\bigr)\)
is not monochromatic. Hence \(P(1)\) holds.

\paragraph{Inductive step.}
Assume \(P(i)\) holds for some \(1\le i<m\), and fix a tail
\((\boldsymbol{x}_{i+2},\ldots,\boldsymbol{x}_{m})\in \mathcal{X}_{i+2}\times\cdots\times \mathcal{X}_{m}\).
For each choice of \(\boldsymbol{x}_{i+1}\in \mathcal{X}_{i+1}\), by \(P(i)\) (applied to the tail
\((\boldsymbol{x}_{i+1},\ldots,\boldsymbol{x}_{m})\)) there exists at least one witnessing tuple
\((\boldsymbol{x}_{1},\ldots,\boldsymbol{x}_{i})\in \mathcal{B}\times \mathcal{X}_{2}\times\cdots\times \mathcal{X}_{i}\)
such that, for every \(1\le j\le i\),
\[
S_{j}\bigl((\boldsymbol{v}_{j},\boldsymbol{x}_{1},\ldots,\boldsymbol{x}_{i},\boldsymbol{x}_{i+1},\ldots,\boldsymbol{x}_{m});\varepsilon\bigr)
\ \text{is not monochromatic.}
\]
Fix an arbitrary choice of one such witnessing tuple for each \(\boldsymbol{x}_{i+1}\),
and define an auxiliary coloring
\[
\gamma_{\boldsymbol{x}_{i+2},\ldots,\boldsymbol{x}_{m}}:\mathcal{X}_{i+1}\to
\mathcal{B}\times \mathcal{X}_{2}\times\cdots\times \mathcal{X}_{i}
\]
by letting \(\gamma_{\boldsymbol{x}_{i+2},\ldots,\boldsymbol{x}_{m}}(\boldsymbol{x}_{i+1})\) be the chosen witnessing tuple associated to \(\boldsymbol{x}_{i+1}\). By the pigeonhole principle, some color class \(\mathcal{C}\subseteq \mathcal{X}_{i+1}\) satisfies
\[
|\mathcal{C}|\ \ge\ \frac{|\mathcal{X}_{i+1}|}{\,|\mathcal{B}|\cdot\prod_{j=2}^{i}|\mathcal{X}_{j}|}
\ \ge\ \frac{|\mathcal{X}_{i+1}|}{\,|\mathcal{B}|\cdot\prod_{j=2}^{i} c(\mathcal{B})^{\,n_{j}}}
\ \ge\ |\mathcal{X}_{i+1}|\cdot(1+\delta(\mathcal{B}))^{-n_{i+1}},
\]
where we used \( |\mathcal{X}_j|\le c(\mathcal{B})^{\,n_j} \) for \(2\le j\le i\), and the choice of \(n_{i+1}\) ensuring
\( |\mathcal{B}| \cdot \prod_{j=2}^{i} c(\mathcal{B})^{\,n_j} \le (1+\delta(\mathcal{B}))^{\,n_{i+1}} \).
By the super-Ramsey property of \(\mathcal{X}_{i+1}\) for the pattern \(\mathcal{B}\),
the set \(\mathcal{C}\) contains a congruent copy \(\mathcal{B}'\subseteq \mathcal{X}_{i+1}\) of \(\mathcal{B}\).

By construction of \(\gamma_{\boldsymbol{x}_{i+2},\ldots,\boldsymbol{x}_{m}}\), every point of \(\mathcal{B}'\) is assigned the same
tuple \((\boldsymbol{x}_{1},\ldots,\boldsymbol{x}_{i})=\gamma_{\boldsymbol{x}_{i+2},\ldots,\boldsymbol{x}_{m}}(\boldsymbol{x}_{i+1})\).
Therefore the set
\[
\Bigl\{\,(\boldsymbol{x}_{1},\ldots,\boldsymbol{x}_{i},\boldsymbol{x}_{i+1},\boldsymbol{x}_{i+2},\ldots,\boldsymbol{x}_{m})
\ :\ \boldsymbol{x}_{i+1}\in \mathcal{B}' \Bigr\}
\]
is a congruent copy of \(\mathcal{B}\) sitting in the \((i+1)\)-st factor \(\mathbb{P}_{i+1}\), with all other
coordinates fixed. Applying Claim~\ref{claim:NotMono} with \(h=i+1\) (which depends only on the
isometry type of \(\mathcal{B}\)), we obtain some \(\boldsymbol{x}_{i+1}\in \mathcal{B}'\) such that
the \(\varepsilon\)-sphere
\[
S_{i+1}\bigl((\boldsymbol{v}_{i+1},\boldsymbol{x}_{1},\ldots,\boldsymbol{x}_{i},\boldsymbol{x}_{i+1},\ldots,\boldsymbol{x}_{m});\varepsilon\bigr)
\]
is not monochromatic. Together with the inductive hypothesis (which already guarantees the same
for \(j=1,\ldots,i\)), this yields \(P(i+1)\).

By induction, the claim holds for all \(i\in[m]\). This finishes the proof.

    \end{poc}

    By~\cref{claim:Iteratively}, \(P(m)\) holds. Taking the tail to be empty, we obtain a tuple
\((\boldsymbol{x}_{1},\ldots,\boldsymbol{x}_{m})\in\mathcal{G}=\mathcal{B}\times \mathcal{X}_{2}\times\cdots\times \mathcal{X}_{m}\)
such that, for every \(i\in[m]\), the \(\varepsilon\)-sphere
\[
S_{i}\bigl((\boldsymbol{v}_{i},\boldsymbol{x}_{1},\ldots,\boldsymbol{x}_{m});\varepsilon\bigr)\ \subseteq\ \mathbb{P}_{i}
\]
is not monochromatic under \(\chi\). This furnishes a single choice of fiber points that
works simultaneously for all anchors \((\boldsymbol{v}_{i},\boldsymbol{x}_{1},\ldots,\boldsymbol{x}_{m})\),
as required, and completes the proof.
\end{proof}

\subsection{An alternative proof of canonical Ramsey property for triangles}\label{subsection:NewproofTriangle}

Using \cref{thm:canonical-perturbation}, we can provide a new proof of the canonical property for triangles. For convenience, we restate the theorem we will prove.
\begin{theorem}\label{thm:triangle another proof}
   Let \(\mathcal{T}\) be a triangle and let \(r\) be a positive integer. There exists some \(n_0(\mathcal{T})\) such that for any $n\ge n_0(\mathcal{T})$
    \[
    \mathbb{E}^{n}\overset{r}{\rightarrow}(\mathcal{T};\mathcal{T})_{\mathrm{GR}}.
    \]  
\end{theorem}
Before the proof, we need an additional theorem about the simplex.

\begin{lemma}[\cite{1990JAMS}]\label{lem:simplex perturbation}
    Let $\mathcal{A}$ be a $(d-1)$-dimensional simplex with vertices $\boldsymbol{v}_{1},\dots,\boldsymbol{v}_{d}$, then there exist $\varepsilon'=\varepsilon'(\mathcal{A})>0$ such that for any $0<\varepsilon\le \varepsilon'$ there exists a $(d-1)$-dimensional simplex $\mathcal{A}_{\varepsilon}$ with vertices $\boldsymbol{v}_{1'},\ldots,\boldsymbol{v}_{d'}$ such that $\|\boldsymbol{v}_{i'}-\boldsymbol{v}_{j'}\|^2=\|\boldsymbol{v}_{i}-\boldsymbol{v}_{j}\|^2-2\varepsilon^2$ for every $1\le i<j\le d$.
\end{lemma}
\begin{rmk}
  The existence of \(\varepsilon'(\mathcal{A})\) was shown in~\cite{1990JAMS}.
\end{rmk}

Using \cref{thm:canonical-perturbation} and~\cref{lem:simplex perturbation}, we give a different proof of the canonical Ramsey property for triangles.

\begin{proof}[Proof of \cref{thm:triangle another proof}]
Let a triangle \(\mathcal{T}\) have side lengths \(a,b,c>0\). By~\cref{lem:simplex perturbation}, there exists \(\varepsilon>0\) and a triangle \(\mathcal{T}_{\varepsilon}\) with side lengths
\[
a':=\sqrt{a^{2}-2\varepsilon^{2}},\quad
b':=\sqrt{b^{2}-2\varepsilon^{2}},\quad
c':=\sqrt{c^{2}-2\varepsilon^{2}} .
\]
Fix an embedding of \(\mathcal{T}_{\varepsilon}\) with vertices
\(\mathcal{X}=\{\boldsymbol{v}_{1},\boldsymbol{v}_{2},\boldsymbol{v}_{3}\}\subseteq \mathbb{P}_{0}\cong\mathbb{E}^{2}\).

Apply \cref{thm:canonical-perturbation} with \(m=3\) and anchors \(\mathcal{X}\) selected above.
We obtain integers \(n_{1},n_{2},n_{3}\) such that for every \(r\)-coloring \(\chi\) of the orthogonal product
\[
\mathbb{P}\ :=\ \mathbb{P}_{0}\times\mathbb{P}_{1}\times\mathbb{P}_{2}\times\mathbb{P}_{3},
\qquad \mathbb{P}_{0}\cong\mathbb{E}^{2},\ \ \mathbb{P}_{i}\cong\mathbb{E}^{n_{i}}\ (i=1,2,3),
\]
that contains no monochromatic copy of \(\mathcal{T}\), there exist points
\(\boldsymbol{x}_{i}\in\mathbb{P}_{i}\) so that for each \(i\in\{1,2,3\}\) the \(\varepsilon\)-sphere in the \(i\)-th fiber around
\[
\boldsymbol{p}_{i}:=(\boldsymbol{v}_{i},\boldsymbol{x}_{1},\boldsymbol{x}_{2},\boldsymbol{x}_{3})\in\mathbb{P}
\]
is not monochromatic.

Notice that any triple of points $(\boldsymbol{y}_1,\boldsymbol{y}_2,\boldsymbol{y}_3)$ such that $\boldsymbol{y}_i\in S_i(\boldsymbol{p}_i;\varepsilon)$ can be expressed as
\[
\boldsymbol{y}_{1}
:=\bigl(\boldsymbol{v}_{1}, \boldsymbol{x}_{1}+\varepsilon\boldsymbol{u}_{1},\boldsymbol{x}_{2},\boldsymbol{x}_{3}\bigr)
\in S_{1}(\boldsymbol{p}_{1};\varepsilon),
\]
\[
\boldsymbol{y}_{2}
:=\bigl(\boldsymbol{v}_{2}, \boldsymbol{x}_{1},\boldsymbol{x}_{2}+\varepsilon\boldsymbol{u}_{2},\boldsymbol{x}_{3}\bigr)
\in S_{2}(\boldsymbol{p}_{2};\varepsilon),
\]
\[
\boldsymbol{y}_{3}
:=\bigl(\boldsymbol{v}_{3}, \boldsymbol{x}_{1},\boldsymbol{x}_{2},\boldsymbol{x}_{3}+\varepsilon\boldsymbol{u}_{3}\bigr)
\in S_{3}(\boldsymbol{p}_{3};\varepsilon),
\]
where \(\boldsymbol{u}_{i}\in\mathbb{P}_{i}\) are unit vectors. Since the summands lie in pairwise orthogonal subspaces, there are no cross terms, hence for distinct \(i,j\in [3]\) we have
\[
\|\boldsymbol{y}_{i}-\boldsymbol{y}_{j}\|^{2}
=\|\boldsymbol{v}_{i}-\boldsymbol{v}_{j}\|^{2}
+\|\varepsilon\boldsymbol{u}_{i}\|^{2}
+\|\varepsilon\boldsymbol{u}_{j}\|^{2}
=\|\boldsymbol{v}_{i}-\boldsymbol{v}_{j}\|^{2}+2\varepsilon^{2}.
\]
By construction of \(\mathcal{T}_{\varepsilon}\), the base distances satisfy
\(\|\boldsymbol{v}_{1}-\boldsymbol{v}_{2}\|=a'\), \(\|\boldsymbol{v}_{2}-\boldsymbol{v}_{3}\|=b'\),
\(\|\boldsymbol{v}_{1}-\boldsymbol{v}_{3}\|=c'\), with
\((a')^{2}+2\varepsilon^{2}=a^{2}\), \((b')^{2}+2\varepsilon^{2}=b^{2}\),
\((c')^{2}+2\varepsilon^{2}=c^{2}\).
Therefore
\[
\|\boldsymbol{y}_{1}-\boldsymbol{y}_{2}\|=a,\qquad
\|\boldsymbol{y}_{2}-\boldsymbol{y}_{3}\|=b,\qquad
\|\boldsymbol{y}_{1}-\boldsymbol{y}_{3}\|=c,
\]
which means any such triple \(\boldsymbol{y}_{1},\boldsymbol{y}_{2},\boldsymbol{y}_{3}\) forms a copy of \(\mathcal{T}\).

For \(i\in [3]\), let \(C_{i}\) be the set of colors that actually appear on the sphere
\(S_{i}(\boldsymbol{p}_{i};\varepsilon)\); by~\cref{thm:canonical-perturbation}, \(|C_{i}|\ge 2\).
 It remains to choose $c_i\in C_i$ that are all the same or distinct. This is because there exists \(\boldsymbol{y}_i\in S_i(\boldsymbol{p}_i;\varepsilon)\) for $i\in [3]$ such that $\boldsymbol{y}_i$ has color $c_i$. Then \(\boldsymbol{y}_{1},\boldsymbol{y}_{2},\boldsymbol{y}_{3}\) forms a copy of $\mathcal{T}$ that is either monochromatic or rainbow.
If \(C_{1}\cap C_{2}\cap C_{3}\neq\emptyset\), take a common color on all three spheres to obtain a
monochromatic copy of \(\mathcal{T}\).
Otherwise the triple intersection is empty. Since each \(|C_{i}|\ge 2\), the union
\(\,C_{1}\cup C_{2}\cup C_{3}\,\) must have size at least \(3\). Hence the family \(\{C_{1},C_{2},C_{3}\}\) satisfies Hall’s condition
\(|\bigcup_{i\in I}C_{i}|\ge |I|\) for every \(I\subseteq\{1,2,3\}\), so there exist distinct
representatives \(c_{1}\in C_{1},c_{2}\in C_{2},c_{3}\in C_{3}\), yielding a rainbow copy of \(\mathcal{T}\).

Therefore every \(r\)-coloring of \(\mathbb{P}\) contains a monochromatic or a rainbow copy of \(\mathcal{T}\). Setting \(n_{0}(\mathcal{T})=2+n_{1}+n_{2}+n_{3}\) completes the proof.
\end{proof}

\subsection{Canonical property for certain \(3\)-dimensional simplices via perturbation theorem}
Using \cref{thm:canonical-perturbation}, we prove a canonical Ramsey theorem for
3-dimensional simplices (tetrahedra) under a natural geometric separation condition.

\begin{theorem}\label{thm:conditioned-simplex-proof}
Let \(\mathcal{T}\) be a tetrahedron such that the largest height of \(\mathcal{T}\)
exceeds the smallest circumradius among its four triangular faces.
Then for every \(r\) there exists \(n_0=n_0(\mathcal{T})\) such that for all \(n\ge n_0\),
\[
\mathbb{E}^{n}\ \overset{r}{\longrightarrow}\ (\mathcal{T};\mathcal{T})_{\mathrm{GR}}.
\]
\end{theorem}

\paragraph{Overview of the proof of~\cref{thm:conditioned-simplex-proof}.} The proof is slightly complicated, especially the construction of some auxiliary configurations. Therefore we plan to first sketch the proof here.
We work under the geometric separation hypothesis that the largest height is larger than the smallest facial circumradius. The engine is the canonical perturbation~\cref{thm:canonical-perturbation}: after embedding into a suitable orthogonal product, we may choose one anchor in each fiber so that every prescribed \(\varepsilon\)-sphere about that anchor is not monochromatic. This yields, for any finite seed configuration \(X\), a palette \(C_{\bm{v}}\subseteq[r]\) with \(|C_v|\ge2\) attached to each \(\bm{v}\in X\), acting as a controlled local obstruction.

We first replace the target tetrahedron by a small contracted copy \(\mathcal{T}'\) that preserves the height–circumradius gap. Two geometric building blocks are then established. Proposition~\ref{claim:two congruent simplices} shows that a short hinge with small opening extends along either side to a copy of \(\mathcal{T}'\). Proposition~\ref{claim:type A} creates a dense micro-configuration: from one face we pick three apices and a fourth above them, realizing four tightly linked copies of \(\mathcal{T}'\). In particular, the existence of such configuration relies on the condition on the circumradius and height in this theorem.

Using these ingredients, we assemble a finite configuration \(\mathcal{X}\) consisting of numerous copies of \(\mathcal{T}'\) glued along polygonal links. The resulting \(\mathcal{X}\) is large and structurally intricate, which is, however, the linchpin of our argument. For this reason we prioritize a complete specification of \(\mathcal{X}\) over brevity and readability. Once \(\mathcal{X}\) is fixed, applying \cref{thm:canonical-perturbation} to \(\mathcal{X}\) equips every anchor with a nontrivial palette. The classification lemma (\cref{lem:classification}) reduces palette quadruples on a tetrahedron to types \(\mathfrak{A}\) or \(\mathfrak{B}\). A five-point transmission rule and a link-consistency principle (\cref{claim:type B}, \cref{claim:linked simplices}) propagate type along links; a tailored sub-configuration \(\mathcal{X}_1\) eliminates \(\mathfrak{B}\) (\cref{claim:no type B}); a second gluing then contradicts \(\mathfrak{A}\) (\cref{claim:no type A}). Hence a monochromatic or rainbow copy of \(\mathcal{T}\) must occur in high enough dimension, proving \cref{thm:conditioned-simplex-proof}.

We then provide the full proof.
\begin{proof}[Proof of \Cref{thm:conditioned-simplex-proof}]
Let \(\mathcal{T}\) be the target \(3\)-dimensional simplex such that the largest height \(H_{\max}(\mathcal{T})\) of \(\mathcal{T}\)
exceeds the smallest circumradius \(\rho_{\min}(\mathcal{T})\) among its four triangular faces. Let \(\chi\) be an \(r\)-coloring of \(\mathbb{E}^{n}\) such that there is neither monochromatic nor rainbow copy of \(\mathcal{T}.\)

By \cref{lem:simplex perturbation}, there exists a contracted tetrahedron $\mathcal{T}_{\varepsilon}$ for $\varepsilon\in(0,\varepsilon_0].$ Since circumradius and a certain height of a tetrahedron are all continuous functions of side length, there exists a proper $\varepsilon$ such that $\mathcal{T}_{\varepsilon}$ still satisfies the condition in \cref{thm:conditioned-simplex-proof}, that is, $H_{\max}(\mathcal{T}_{\varepsilon})>\rho_{\min}(\mathcal{T}_{\varepsilon})$. Denote this tetrahedron $\mathcal{T}':=\mathcal{T}_{\varepsilon}:=\{\bm{a}_{1},\bm{a}_{2},\bm{a}_{3},\bm{a}_{4}\}$, where $H_{\max}(\mathcal{T}')>\rho_{\min}(\mathcal{T}')$. 

{We now need to very carefully select a set \(\mathcal{X}\) of size \(m\) to apply~\cref{thm:canonical perturbation}. The way of selection is complicated, we need a series of important geometric facts. The first fact shows that in $\mathbb{E}^4$, two simplices congruent to $\mathcal{T}'$ can be glued together along a face freely.
\begin{prop}\label{prop:two congruent simplices}
    Let $\bm{a}\bm{b}\bm{c}\bm{d}$ and $\bm{a}'\bm{b}\bm{c}\bm{d}$ be two copies of $\mathcal{T}'$ in $\mathbb{E}^4,$ with $\|\bm{a}'-\bm{b}\|=\|\bm{a}-\bm{b}\|=\|\bm{a}_1-\bm{a}_2\|.$ Let $\theta$ be the induced angle between the vector $\bm{a}_2-\bm{a}_1$ and the plane spanned by \(\{\boldsymbol{a}_2,\boldsymbol{a}_3,\boldsymbol{a}_4\}\). Then $\angle\bm{a}\bm{b}\bm{a}'$ can attain any value in $(0,2\theta].$
\end{prop}
\begin{proof}[Proof of \cref{prop:two congruent simplices}]
Let $\bm{h}$ be the projection of $\bm{a}_1$ on the plane spanned by $\{\bm{a}_2,\bm{a}_3,\bm{a}_4\}$, with $d=\|\bm{a}-\bm{h}\|$; and let $\theta$ be the induced angle between the vector $\bm{a}_2-\bm{a}_1$ and the plane spanned by \(\{\boldsymbol{a}_2,\boldsymbol{a}_3,\boldsymbol{a}_4\}\). Fix a triangle $\triangle\bm{bcd}$ on the plane $\{(x,y,z,w):z=w=0\}\subseteq\mathbb{E}^4$ congruent to $\triangle \bm{a}_2\bm{a}_3\bm{a}_4,$ with $\bm{h}'=(h_1,h_2,0,0)$ at the corresponding location of $\bm{h}$. The set of points $\bm{a}$ such that $\bm{a}\bm{b}\bm{c}\bm{d}$ forming a congruent copy of $\mathcal{T}'$ should satisfy two properties: i) $\|\bm{a}-\bm{h}\|=d$, and ii) $\bm{a}-\bm{h}$ is perpendicular to the plane $\{(x,y,z,w):z=w=0\}.$ To be specific,
\[\{\bm{a}:\bm{a}\bm{b}\bm{c}\bm{d}\cong\mathcal{T}'\}=\{(h_1,h_2,d\cos\gamma,d\sin\gamma):\gamma\in[0,2\pi)\}.\]

Pick two distinct points $\bm{a}=(h_1,h_2,d\cos\gamma,d\sin\gamma)$ and $\bm{a}'=(h_1,h_2,d\cos\gamma',d\sin\gamma')$ from the above set. We calculate $\angle\bm{a}\bm{b}\bm{a}':$
  \[\cos\angle \bm{a}\bm{b}\bm{a}'=\frac{(\bm{b}-\bm{a})\cdot (\bm{b}-\bm{a}')}{\left\|\bm{b}-\bm{a}\right\|\left\|\bm{b}-\bm{a}'\right\|}=\frac{\|\bm{b}-\bm{h}\|^2+(\bm{h}-\bm{a})\cdot(\bm{h}-\bm{a}')}{\|\bm{b}-\bm{h}\|^2+\|\bm{h}-\bm{a}\|^2}=\frac{\|\bm{b}-\bm{h}\|^2+\|\bm{h}-\bm{a}\|^2\cos(\gamma-\gamma')}{\|\bm{b}-\bm{h}\|^2+\|\bm{h}-\bm{a}\|^2}.\]
Since $\bm{a}\neq\bm{a}'$, we have $\gamma\neq\gamma'$ and $\cos(\gamma-\gamma')$ takes every value in $[-1,1)$. Observe that 
    \[\cos\theta=\frac{\|\bm{b}-\bm{h}\|}{\|\bm{a}-\bm{b}\|}=\frac{\|\bm{b}-\bm{h}\|}{\sqrt{\|\bm{a}-\bm{h}\|^2+\|\bm{b}-\bm{h}\|^2}},\]
    and thus \[\cos2\theta=2\cos^2\theta-1=\frac{\|\bm{b}-\bm{h}\|^2-\|\bm{a}-\bm{h}\|^2}{{\|\bm{a}-\bm{h}\|^2+\|\bm{b}-\bm{h}\|^2}}.\]
It follows that $\angle\bm{a}\bm{b}\bm{a}'$ takes every value in $(0,2\theta]$.
\end{proof}

One can immediately obtain the following corollary of~\cref{prop:two congruent simplices} for future application.
\begin{cor}\label{claim:two congruent simplices}
     Let $\boldsymbol{a},\boldsymbol{b},\boldsymbol{a}'\in \mathbb{E}^{4}$ satisfy $\|\boldsymbol{a}-\boldsymbol{b}\|=\|\boldsymbol{b}-\boldsymbol{a}'\|=\|\bm{a}_1-\bm{a}_2\|$, and $\angle \bm{a}\bm{b}\bm{a}'/2$ be smaller than \(\theta\). 
     Then there exist two points $\bm{c}$ and $\bm{d}$ such that both $\bm{a}\bm{b}\bm{c}\bm{d}$ and $\bm{a}'\bm{b}\bm{c}\bm{d}$ are congruent copies of $\mathcal{T'}$. 
\end{cor}

The next fact shows that in $\mathbb{E}^5$, distinct copies of \(\mathcal{T}'\) can be arranged comparatively densely, which relies on the condition on the circumradius and height, $\rho_{\min}(\mathcal{T}')< H_{\max}(\mathcal{T}').$
\begin{prop}\label{claim:type A}
   For any \(n\ge 5\), there exist points \(\bm{z},\bm{y}_1,\bm{y}_2,\bm{y}_3,\bm{x}_1,\bm{x}_2,\bm{x}_3\in\mathbb{E}^{n}\) such that \(\bm{z}\bm{y}_1\bm{y}_2\bm{y}_3\) and, for each \(i\in\{1,2,3\}\), \(\bm{y}_i\bm{x}_1\bm{x}_2\bm{x}_3\) are tetrahedra, all congruent to \(\mathcal{T}'\).
\end{prop}
\begin{proof}[Proof of~\cref{claim:type A}]
    Let $\bm{x}_1,\bm{x}_2,\bm{x}_3$ form the triangle congruent to the face of $\mathcal{T}'$ supporting the largest height $H_{\max}(\mathcal{T}')$. Then $\{\bm{y}:\bm{x}_1\bm{x}_2\bm{x}_3\bm{y}\cong \mathcal{T}'\}$ is an $(n-2)$-dimensional sphere of radius $H_{\max}(\mathcal{T}').$ Hence, there exist three points $\bm{y}_1,\bm{y}_2,\bm{y}_3$ on it, forming the face of minimum circumradius of $\mathcal{T}'$ as $H_{\max}(\mathcal{T}')>\rho_{\min}(\mathcal{T}')$. Finally, fix \(\bm{y}_1,\bm{y}_2,\bm{y}_3\). The locus of points \(\bm{z}\) with \(\bm{y}_1\bm{y}_2\bm{y}_3\bm{z}\cong\mathcal{T}'\) is again an \((n-3)\)-dimensional sphere (by the same intersection-of-spheres argument), which is nonempty when \(n\ge 5\). Pick any \(\bm{z}\) on this sphere. Then \(\bm{z}\bm{y}_1\bm{y}_2\bm{y}_3\) and each \(\bm{y}_i\bm{x}_1\bm{x}_2\bm{x}_3\) are congruent to \(\mathcal{T}'\), as required.
\end{proof}

\begin{figure}[htbp]
\centering
\begin{minipage}{0.45\textwidth}
\centering
\begin{tikzpicture}
 \coordinate (a') at (0,-0.5);
 \filldraw[black] (a') circle (1pt) node [below left]
{$\bm{a}$}; 
 \coordinate (b') at (4.5,4);
 \filldraw[black] (b') circle (1pt) node [right]
{$\bm{b}$};
 \coordinate (c') at (2.7,-3.2);
 \filldraw[black] (c') circle (1pt) node [below]
{$\bm{c}$};
 \coordinate (d') at (4,-0.23);
 \filldraw[black] (d') circle (1pt) node [above right]
{$\bm{d}$};
\coordinate (h) at (3.55,-0.6);
 \coordinate (a'') at (6.3,-1.1);
 \filldraw[black] (a'') circle (1pt) node [below]
{$\bm{a'}$};
\draw[blue,dashed] (3.6,-0.5) ellipse (3.6 and 0.85);
\draw[black,very thick](c')--(a')--(b')--(c')--(d')--(b');
\draw[black,very thick,dashed](a')--(d');
\draw[black,very thick,dashed](a'')--(b');
\draw[black,very thick,dashed](a'')--(c');
\draw[black,very thick,dashed](a'')--(d');
 
\end{tikzpicture}
\caption{The configuration constructed in Prop 4.8}
\end{minipage}
\hfill
\begin{minipage}{0.45\textwidth}
\centering

\begin{tikzpicture}
 \coordinate (x_1) at (0,0);
 \filldraw[black] (x_1) circle (1pt) node [below]
{$\bm{x}_1$};
\coordinate (x_2) at (3,-0.5);
 \filldraw[black] (x_2) circle (1pt) node [below]
{$\bm{x}_2$};
\coordinate (x_3) at (4.5,1);
 \filldraw[black] (x_3) circle (1pt) node [below right]
{$\bm{x}_3$};
\coordinate (y_1) at (-0.1,3.3);
 \filldraw[black] (y_1) circle (1pt) node [below left]
{$\bm{y}_1$};
\coordinate (y_2) at (1.3,5.1);
 \filldraw[black] (y_2) circle (1pt) node [above right]
{$\bm{y}_2$};
\coordinate (y_3) at (4.2,3.9);
 \filldraw[black] (y_3) circle (1pt) node [below right]
{$\bm{y}_3$};
\coordinate (z) at (1.5,7.5);
 \filldraw[black] (z) circle (1pt) node [right]
{$\bm{z}$};
\draw[black,very thick] (x_1)--(x_2)--(x_3);
\draw[black,very thick,dashed](x_1)--(x_3);
\draw[blue,very thick,dashed] (x_1)--(y_1);
\draw[blue,very thick,dashed] (x_2)--(y_1);
\draw[blue,very thick,dashed] (x_3)--(y_1);
\draw[red,very thick,dashed] (x_1)--(y_2);
\draw[red,very thick,dashed] (x_2)--(y_2);
\draw[red,very thick,dashed] (x_3)--(y_2);
\draw[green,very thick,dashed] (x_1)--(y_3);
\draw[green,very thick,dashed] (x_2)--(y_3);
\draw[green,very thick,dashed] (x_3)--(y_3);
\draw[black,very thick] (y_1)--(z)--(y_3)--(y_1);
\draw[black,very thick,dashed] (y_2)--(y_1);
\draw[black,very thick,dashed] (y_2)--(y_3);
\draw[black,very thick,dashed] (y_2)--(z);

\end{tikzpicture}
\caption{The configuration constructed in Prop 4.9}
\end{minipage}
\end{figure}

The pivotal ingredient of the proof is the construction of \(\mathcal{X}\). Since the selection is technically intricate, we describe it in full detail.

\begin{defn}[Configuration \(\mathcal{X}\)]\label{defn:construction}
The construction of \(\mathcal{X}\) mixes various components. Given a tetrahedron $\mathcal{T}'$ formed by four points $\bm{a}_1,\bm{a}_2,\bm{a}_3,\bm{a}_4$, let $\bm{a}_1\bm{a}_2\bm{a}_3$ form the face supporting the largest height.

\paragraph{Link between two simplices.}
    Firstly, we introduce the \emph{link} between two congruent simplices. Let $\bm{a}\bm{b}\bm{c}\bm{d}\simeq \bm{a}'\bm{b}'\bm{c}'\bm{d}'\simeq\mathcal{T}'.$ The construction is conducted in three steps.
    \begin{enumerate}
        \item Connect $\bm{b}$ and $\bm{b}'$ by a polygonal path whose edges have length $\|\bm{a}-\bm{b}\|,$ namely $\bm{b}=\bm{x}_1,\bm{a}=\bm{x}_2,\bm{x}_3,\dots,\bm{x}_{k-2},\bm{a}'=\bm{x}_{k-1},\bm{b}'=\bm{x}_k$. For each \(2\le i\le k-1\), choose an integer $k_i\geq 1$ and insert a sequence of points $\bm{y}_{i,1}=\bm{x}_{i-1},\bm{y}_{i,2},\dots,\bm{y}_{i,k_i}=\bm{x}_{i+1}$ such that $\|\bm{y}_{i,j}-\bm{x}_i\|=\|\bm{a}-\bm{b}\|$ and each angle $\angle \bm{y}_{i,j-1}\bm{x}_{i}\bm{y}_{i,j}$ satisfies the condition of \cref{claim:two congruent simplices}.  
        \item Apply \cref{claim:two congruent simplices} to every three points $\bm{y}_{i,j},\bm{x}_i,\bm{y}_{i,j+1}$ to obtain two new points $\bm{z}_{i,j,1}$ and $\bm{z}_{i,j,2}$ such that the polygons $\bm{y}_{i,j}\bm{x}_i\bm{z}_{i,j,1}\bm{z}_{i,j,2}$ and  $\bm{y}_{i,j+1}\bm{x}_i\bm{z}_{i,j,1}\bm{z}_{i,j,2}$ are congruent to $\mathcal{T}'.$
        \item Repeat the above steps to connect $\bm{d}$ and $\bm{d}'$ using the step length $\|\bm{c}-\bm{d}\|.$ This produces another polygonal path 
    \[\bm{c}=\bm{x}'_1,\bm{d}=\bm{x}'_2,\bm{x}'_3,\dots,\bm{x}'_{\ell-1}=\bm{d}',\bm{x}'_{\ell}=\bm{c},\]
    together with subdivided points $\bm{y}_{i,j}'$ and vertices
    \(\bm z'_{i,j,1},\bm z'_{i,j,2}\) that again generate congruent copies of \(\mathcal{T}'\). 
    \end{enumerate}
    The link of $\bm{abcd}$ and $\bm{a}'\bm{b}'\bm{c}'\bm{d}'$ consist of all the points produced above, that is,
    \begin{align*}
    \mathcal{L}(\bm{abcd},\bm{a}'\bm{b}'\bm{c}'\bm{d}')=&\{\bm{a},\bm{b},\bm{c},\bm{d},\bm{a}',\bm{b}',\bm{c}',\bm{d}'\}\cup\{\bm{x}_i\}\cup\{\bm{y}_{i,j}\}\cup\{\bm{z}_{i,j,1},\bm{z}_{i,j,2}\}\\ &\cup\{\bm{x}_i'\}\cup\{\bm{y}_{i,j}'\}\cup\{\bm{z}_{i,j,1}',\bm{z}_{i,j,2}'\},
    \end{align*}
where we omit all of the ranges of indices above for clarity.

\paragraph{Gluing the links.}
    Let $\bm{a},\bm{b},\bm{c},\bm{d}$ form a congruent copy of $\mathcal{T}'$. Secondly, we construct a useful configuration $\mathcal{X}_1$ based on the link of two simplices. This includes five steps.
    
    \begin{enumerate}
        \item 
Connect \(\bm b\) and \(\bm c\) by a polygonal path each of whose edges has length \(\|\bm a-\bm b\|\).
Similarly, connect \(\bm c\) to \(\bm d\) and \(\bm d\) to \(\bm a\).
This yields an equilateral polygon
\(\bm u_{1}\ldots \bm u_{\ell_{1}}\,\bm u_{\ell_{1}+1}\ldots \bm u_{\ell_{2}}\bm u_{\ell_{2}+1}\ldots \bm u_{\ell_{3}}\),
where $\bm{a}=\bm{u}_1,\bm{b}=\bm{u}_2,\bm{c}=\bm{u}_{\ell_1},\bm{d}=\bm{u}_{\ell_2}$, \(0\le \ell_{1}<\ell_{2}<\ell_{3}\) and every consecutive pair of vertices in the sequence is at distance \(\|\bm a-\bm b\|\).

        \item For each \(i\in [\ell_{3}]\), we then choose an integer $m_i\geq 1$ and insert a sequence of points $\bm{u}_{i-1}=\bm{z}_{i,1},\bm{z}_{i,2},\ldots,\bm{z}_{i,m_{i}-1},\bm{z}_{i,m_i}=\bm{u}_{i+1}$ such that $\|\bm{z}_{i,k}-\bm{u}_i\|=\|\bm{a}-\bm{b}\|$ and each angle $\angle \bm{z}_{i,j-1}\bm{u}_i\bm{z}_{i,j}$ satisfies the condition of \cref{claim:two congruent simplices}. 
        \item For each $i\in [\ell_{3}],j\in [m_{i}-1]$, we then apply \cref{claim:two congruent simplices} to each triple of points $\{\bm{z}_{i,j},\bm{u}_i,\bm{z}_{i,j+1}\}$ to obtain new sequence of points $\bm{x}_{i,j,1}$ and $\bm{x}_{i,j,2}$ such that the polygons $\bm{z}_{i,j}\bm{u}_i\bm{x}_{i,j,1}\bm{x}_{i,j,2}$ and  $\bm{z}_{i,j+1}\bm{u}_i\bm{x}_{i,j,1}\bm{x}_{i,j,2}$ are congruent to $\mathcal{T}'.$
        \item Let \(\phi:=\sum_{i=1}^{\ell_3}2(m_i-1)\) be the total number of tetrahedra congruent to $\mathcal{T}'$ appeared after the above steps.
        Label them as $\mathcal{T}'_1,\dots,\mathcal{T}'_\phi$. Link $\mathcal{L}(\mathcal{T}'_i,\mathcal{T}'_{i+1})$ for $i\in[\phi-1]$.
        \item Repeat the above four steps with the roles of \((\bm a,\bm b,\bm c,\bm d)\) replaced by
    \((\bm c,\bm d,\bm a,\bm b)\), using the step length \(\|\bm c-\bm d\|\).
    This produces another closed equilateral polygon
    \[
    \bm c,\ \bm d,\ \bm u'_{1},\dots,\bm u'_{k_{1}},\ \bm a,\
    \bm u'_{k_{1}+1},\dots,\bm u'_{k_{2}},\ \bm b,\
    \bm u'_{k_{2}+1},\dots,\bm u'_{k_{3}},\ \bm c,
    \]
    together with subdivided points \(\bm z'_{i,j}\), points
    \(\bm x'_{i,j,1},\bm x'_{i,j,2}\) that again generate congruent copies of \(\mathcal{T}'\), namely $\mathcal{T}_{\phi+1}',\dots,\mathcal{T}'_{\varphi}$, and the consecutive links $\mathcal{L}(\mathcal{T}'_i,\mathcal{T}'_{i+1}).$
    \end{enumerate}
    Let \(\mathcal{X}_{1}\) consist of all of the points produced from the above five steps, formally we define 
    \begin{align*}
        \mathcal{X}_{1}(\bm{a},\bm{b},\bm{c},\bm{d}):=&\{\bm{a},\bm{b},\bm{c},\bm{d}\}\cup \{\boldsymbol{u}_{i}\}\cup\{\bm{z}_{i,j}\}\cup\{\bm{x}_{i,j,1},\bm{x}_{i,j,2}\}\cup\mathcal{L}(\mathcal{T}'_i,\mathcal{T}'_{i+1})\\
&         \cup \{\boldsymbol{u}'_{i}\}\cup\{\bm{z}'_{i,j}\}\cup\{\bm{x}'_{i,j,1},\bm{x}'_{i,j,2}\},
    \end{align*}
   where we also omit the ranges of indices for clarity. 

\paragraph{Completing the construction of \(\mathcal{X}\).}   
    
     Recall that tetrahedron $\mathcal{T}'$ is formed by four points $\bm{a}_1,\bm{a}_2,\bm{a}_3,\bm{a}_4$, let $\bm{a}_1\bm{a}_2\bm{a}_3$ form the face supporting the largest height. Notice that there exists some point \(\bm{x}\) such that $\bm{a_1}\bm{a}_2\bm{x}\bm{a}_3$ is a parallelogram, we then set  $d:=\|\bm{a}_1-\bm{x}\|$. Connect $\bm{a}_1$ and $\bm{a}_2$ by a polygonal path $\bm{b}_0=\bm{a}_1,\bm{b}_1,\dots,\bm{b}_k,\bm{a}_2=\bm{b}_{k+1}$ such that $\|\bm{b}_i-\bm{b}_{i+1}\|=d$. Extend each edge $\bm{b}_i\bm{b}_{i+1}$ to a parallelogram $\bm{b}_i\bm{c}_{i,1}\bm{b}_{i+1}\bm{c}_{i,2}$ congruent to $\bm{a}_1\bm{a}_2\bm{x}\bm{a}_3$.
     
    For each \(0\le i\le k\), notice that the triangles $\bm{b}_i\bm{c}_{i,1}\bm{c}_{i,2}$ and $\bm{b}_{i+1}\bm{c}_{i,1}\bm{c}_{i,2}$ are all congruent to the face $\bm{a}_1\bm{a}_2\bm{a}_3$ supporting the largest height. Then applying~\cref{claim:type A}, we can find $ \bm{z}_i,\bm{y}_{i,1},\bm{y}_{i,2},\bm{y}_{i,3}$ and $\bm{z}_i',\bm{y}_{i,1}',\bm{y}_{i,2}',\bm{y}_{i,3}'$ such that all of the following \(4\)-tuples form a copy of \(\mathcal{T}'\):
    \begin{enumerate}
        \item $\{\bm{b}_i,\bm{c}_{i,1},\bm{c}_{i,2},\bm{y}_{i,j}\}_{j\in [3]}$;
        \item $\{\bm{b}_{i+1},\bm{c}_{i,1},\bm{c}_{i,2},\bm{y}_{i,j}\}_{j\in [3]}$;
        \item $\{\bm{z}_i,\bm{y}_{i,1},\bm{y}_{i,2},\bm{y}_{i,3}\}$;
        \item \(\{\bm{z}_i',\bm{y}_{i,1}',\bm{y}_{i,2}',\bm{y}_{i,3}'\}\).
    \end{enumerate}
Fix all of the above selected points. Now for each copy of \(\mathcal{T}'\) with points \(\bm{a}_{1},\bm{a}_{2},\bm{a}_{3},\bm{a}_{4}\), we define the following configuration based on $\mathcal{T}$ with respect to the first two points \(\{\bm{a}_{1},\bm{a}_{2}\}\) as 
     \begin{align*}
         \mathcal{K}(\mathcal{T}';\bm{a}_1,\bm{a}_2)=&\{\bm{a}_i\}\cup\{\bm{b}_i\}\cup\{\bm{c}_{i,1},\bm{c}_{i,2}\}\cup\{\bm{z}_i,\bm{z}_i'\}\cup\{y_{i,j},y_{i,j}'\}\cup \{\mathcal{X}_1(\bm{z}_i,\bm{y}_{i,1},\bm{y}_{i,2},\bm{y}_{i,3})\}\\ &\cup \{\mathcal{X}_1(\bm{z}_i',\bm{y}_{i,1}',\bm{y}_{i,2}',\bm{y}_{i,3}')\}
         \cup \{\mathcal{X}_1(\bm{y}_{i,k},\bm{b}_{i},\bm{c}_{i,1},\bm{c}_{i,2})\}\cup \{\mathcal{X}_1(\bm{y}_{i,k}',\bm{b}_{i+1},\bm{c}_{i,1},\bm{c}_{i,2})\},
     \end{align*}
     where we omit the ranges for indices again for readability. We also define \(\mathcal{K}(\mathcal{T}';\bm{a}_2,\bm{a}_3)\) similarly. Finally, take \[\mathcal{X}=\mathcal{K}(\mathcal{T}';\bm{a}_1,\bm{a}_2)\cup\mathcal{K}(\mathcal{T}';\bm{a}_2,\bm{a}_3)\cup \mathcal{X}_1(\bm{a}_1,\bm{a}_2,\bm{a}_3,\bm{a}_4).\]

\end{defn}

}

Recall that \(\chi\) is an \(r\)-coloring of \(\mathbb{E}^{n}\) such that there is neither monochromatic nor rainbow copy of \(\mathcal{T}.\) The following series of claims together yield a contradiction. From now on, we use the notation in \cref{defn:construction} unless stated otherwise. 

\begin{claim}
    There exists a copy of \(\mathcal{T}'\subseteq \mathcal{X}.\) 
\end{claim}
\begin{poc}
    By the definition of $\mathcal{X},$ we have $\bm{a}_1,\bm{a}_2,\bm{a}_3,\bm{a}_4\in \mathcal{X}$ and $\bm{a}_1\bm{a}_2\bm{a}_3\bm{a}_4\cong\mathcal{T}'.$
\end{poc}

Temporarily relabel the points in $\mathcal{X}$ as $\{\bm{v}_1,\dots,\bm{v}_m\}$ to apply~\cref{thm:canonical perturbation} with \(m=|\mathcal{X}|\) and \(\mathbb{P}:=\mathbb{P}_{0}\times\mathbb{P}_{1}\times\cdots\times\mathbb{P}_{m}\cong \mathbb{E}^{n_{0}}\times\mathbb{E}^{n_{1}}\times\cdots\times\mathbb{E}^{n_{m}}\). Let $C_{i}$ be the set of colors appearing in the sphere $S_{i}(\bm{p}_{i},\varepsilon).$ By~\cref{thm:canonical perturbation}, we have \(|C_{i}|\ge 2\) for each \(i\in [m].\)

Let \(\{\bm{v}_{i_1},\bm{v}_{i_2},\bm{v}_{i_3},\bm{v}_{i_4}\}\subseteq \mathcal{X}\) form a copy of \(\mathcal{T}'\), where \(\boldsymbol{v}_{i_{j}}\in\mathbb{P}_{i_{j}}\).

\begin{claim}\label{lem:classification}
Let \(C_{i_1},C_{i_2},C_{i_3},C_{i_4}\) with \(|C_{i_j}|\ge 2\) for all \(j\).
Assume there is no rainbow choice of distinct representatives from
\(\{C_{i_1},C_{i_2},C_{i_3},C_{i_4}\}\),
and there is no color contained in all four sets.
Then, up to relabeling the colors and the indices \(i_1,i_2,i_3,i_4\), the quadruple
\((C_{i_1},C_{i_2},C_{i_3},C_{i_4})\) is in exactly one of the following two forms:
\begin{enumerate}
    \item Type \(\mathfrak{A}\): \(C_{i_1}=\{1,2\}\), \(C_{i_2}=\{2,3\}\), \(C_{i_3}=\{1,3\}\), and \(C_{i_4}\subseteq\{1,2,3\}\);
    \item Type \(\mathfrak{B}\): \(C_{i_1}=C_{i_2}=C_{i_3}=\{1,2\}\) and \(C_{i_4}\supseteq\{3,4\}\).
\end{enumerate}
\end{claim}
\begin{poc}
Write \(\mathcal{C}=\{C_{i_1},C_{i_2},C_{i_3},C_{i_4}\}\) and set
\(U:=C_{i_1}\cup C_{i_2}\cup C_{i_3}\cup C_{i_4}\).
We consider two cases according to \(|U|\).
\paragraph{Case \(|U|\ge 4\).}  If every triple among the four sets had union of size at least \(3\), then by Hall’s theorem
the family \(\mathcal{C}\) would admit a system of distinct representatives (take one color from
each set, all distinct), that is, a rainbow choice, contrary to the assumption.
Hence there exist distinct indices \(m,k,\ell\) with
\(|C_{i_m}\cup C_{i_k}\cup C_{i_\ell}|\le 2\).
Since \(|C_{i_j}|\ge 2\) for all $j$, we must have
\(C_{i_m}=C_{i_k}=C_{i_\ell}=\{1,2\}\) after relabeling colors.
To avoid a monochromatic choice, the remaining set \(C_{i_r}\) (with
\(\{m,k,\ell,r\}=\{1,2,3,4\}\)) cannot contain \(1\) or \(2\).
Because \(|C_{i_r}|\ge 2\), we get \(C_{i_r}\supseteq\{3,4\}\), which is precisely type \(\mathfrak{B}\).   

\paragraph{Case \(|U|\le 3\).}   If \(|U|\le 2\), then every set in \(\mathcal{C}\) is a \(2\)-subset of the same \(2\)-element ground set,
so some color belongs to all four sets, contradicting the assumption.
Thus \(|U|=3\). Up to relabeling, \(U=\{1,2,3\}\), and the only \(2\)-subsets of \(U\) are
\(\{1,2\},\{2,3\},\{1,3\}\).
If one of these three pairs is absent from \(\mathcal{C}\), say \(\{1,2\}\) does not occur,
then every member of \(\mathcal{C}\) contains color \(3\), again giving a monochromatic choice.
Therefore all three pairs appear among \(C_{i_1},C_{i_2},C_{i_3},C_{i_4}\).
Since \(|U|=3\), necessarily \(C_{i_4}\subseteq\{1,2,3\}\).
Relabeling indices if needed, we obtain exactly the pattern listed in type \(\mathfrak{A}\).

This finishes the proof.
\end{poc}

Recall that there is no rainbow or monochromatic copy of $\mathcal{T}$ under coloring \(\chi\). Next it suffices to show that both cases in~\cref{lem:classification} cannot occur.

\begin{claim}\label{claim:type B}
    If there are five points in $\mathcal{X}$, say $\bm{v}_{i_1},\bm{v}_{i_1'},\bm{v}_{i_2},\bm{v}_{i_3},\bm{v}_{i_4}$, forming the five-point configuration in \cref{claim:two congruent simplices} with $\bm{v}_{i_1}\bm{v}_{i_2}\bm{v}_{i_3}\bm{v}_{i_4},\bm{v}_{i_1'}\bm{v}_{i_2}\bm{v}_{i_3}\bm{v}_{i_4}\cong\mathcal{T}'$, then $C_{i_1}\cap C_{i_2}=\emptyset$ if and only if $C_{i_1'}\cap C_{i_2}=\emptyset.$
\end{claim}
\begin{poc}
    If $C_{i_1}\cap C_{i_2}=\emptyset$, then the tetrahedron $\bm{v}_{i_1}\bm{v}_{i_2}\bm{v}_{i_3}\bm{v}_{i_4}$ belongs to type $\mathfrak{B}$ with $C_{i_3}=C_{i_4}$. Hence, either $|C_{i_2}\cup C_{i_3}\cup C_{i_4}|=2$ or {$|C_{i_2}\cup C_{i_3}\cup C_{i_4}| \ge 4$}. Consequently, the tetrahedron $\bm{v}_{i_1'}\bm{v}_{i_2}\bm{v}_{i_3}\bm{v}_{i_4}$ belongs to type $\mathfrak{B}$ as well, and furthermore $C_{i_1'}\cap C_{i_2}=\emptyset.$ Vice versa.
\end{poc}
From now on, with a slight abuse of notation, we use $C_{\bm{a}}$ for the set of colors on the $\varepsilon$-sphere $S_{\bm{a}}(\bm{p}_{\bm{a}},\varepsilon)$ centered at $\bm{a}.$ 

\begin{claim}\label{claim:linked simplices}
    If two simplices $\bm{a}\bm{b}\bm{c}\bm{d}$ and $\bm{a}'\bm{b}'\bm{c}'\bm{d}'$ congruent to $\mathcal{T}'$ both belong to $\mathcal{X}$ and furthermore their link $\mathcal{L}(\mathcal{T}_1',\mathcal{T}_2')$ belongs to $\mathcal{X}$ as well, then they belong to the same type.
\end{claim} 
\begin{poc}

    Apply \cref{claim:type B} to the five-point structure $\bm{y}_{i,j}\bm{y}_{i,j+1}\bm{x}_i\bm{z}_{i,j,1}\bm{z}_{i,j,2}$, we have that $C_{\bm{y}_{i,j}}\cap C_{\bm{x_i}}=\emptyset$ if and only if $C_{\bm{y}_{i,j+1}}\cap C_{\bm{x_i}}=\emptyset$. Since $\bm{y}_{i,1}=\bm{x}_{i-1}$ and $\bm{y}_{i,k_i}=\bm{x}_{i+1}$, we have $C_{\bm{x}_{i-1}}\cap C_{\bm{x}_i}=\emptyset$ if and only if  $C_{\bm{x}_{i}}\cap C_{\bm{x}_{i+1}}=\emptyset.$ Therefore, $C_{\bm{a}}\cap C_{\bm{b}}=\emptyset$ if and only if $C_{\bm{a}'}\cap C_{\bm{b}'}=\emptyset$. Similarly, $C_{\bm{c}}\cap C_{\bm{d}}=\emptyset$ if and only if $C_{\bm{c}'}\cap C_{\bm{d}'}=\emptyset$. 

    We shall recover their type in the following manner: if both $C_{\bm{a}}\cap C_{\bm{b}}$ and $C_{\bm{c}}\cap C_{\bm{d}}$ are non-empty, then $\{C_{\bm{a}},C_{\bm{b}},C_{\bm{c}},C_{\bm{d}}\}$ belongs to type $\mathfrak{A}$; if one intersection is empty and the other one is not, then $\{C_{\bm{a}},C_{\bm{b}},C_{\bm{c}},C_{\bm{d}}\}$ belongs to type $\mathfrak{B}.$ In either case, these two simplices must belong to the same type.
\end{poc}
\begin{claim}\label{claim:no type B}
    If there are some \(4\)-tuple of points $\{\bm{a},\bm{b},\bm{c},\bm{d}\}\subseteq\mathcal{X}$ such that $\mathcal{X}_1(\bm{a},\bm{b},\bm{c},\bm{d})\subseteq \mathcal{X}$, then $C_{\bm{a}},C_{\bm{b}},C_{\bm{c}},C_{\bm{d}}$ should belong to type $\mathfrak{A}.$
\end{claim}
\begin{poc}
   Assume the contrary that $C_{\bm{a}},C_{\bm{b}},C_{\bm{c}},C_{\bm{d}}$ belong to type $\mathfrak{B}.$ Without of loss of generality, assume that $C_{\bm{a}}=C_{\bm{b}}=C_{\bm{c}}$ and $C_{\bm{a}}\cap C_{\bm{d}}=\emptyset.$ Apply \cref{claim:linked simplices} to linked simplices $\bm{z}_{i,j}\bm{u}_i\bm{x}_{i,j,1}\bm{x}_{i,j,2}$ and $\bm{z}_{i,j+1}\bm{u}_i\bm{x}_{i,j,1}\bm{x}_{i,j,2}$ to show that they belong to the same type. Since $\bm{z}_{1,1}=\bm{u}_1=\bm{a}$ and $\bm{u}_2=\bm{b}$ with $C_{\bm{a}}=C_{\bm{b}}$, it follows that $C_{\bm{u}_i}=C_{\bm{z}_{i,j+1}}$. In particular, $C_{\bm{a}}=C_{\bm{u}_1}=C_{\bm{u}_2}=\dots=C_{\bm{u}_{\ell_3}}=C_{\bm{d}}$, contradicting to $C_{\bm{a}}\cap C_{\bm{d}}=\emptyset.$
\end{poc}
\begin{claim}\label{claim:no type A}
    If twelve points in $\mathcal{X}$, say $\bm{v}_z,\bm{v}_{y_1},\bm{v}_{y_2},\bm{v}_{y_3},\bm{v}_{x_1},\bm{v}_{x_2},\bm{v}_{x_3},\bm{v}_{z'},\bm{v}_{y_1'},\bm{v}_{y_2'},\bm{v}_{y_3'},\bm{v}_{x_3'}$ form two copies of the configuration in \cref{claim:type A} sharing two points $\bm{v}_{x_2}$ and $\bm{v}_{x_3}$, and non of these copies of $\mathcal{T'}$ belongs to type $\mathfrak{B}$, then $C_{\bm{x}_3}=C_{\bm{x}_3'}.$
\end{claim}
\begin{poc}
    For simplices of type $\mathfrak{B}$, the union of two color sets determines its palette, i.e., $C_{\bm{x}_1}\cup C_{\bm{x}_2}=C_{\bm{x}_1}\cup C_{\bm{x}_2}\cup C_{\bm{x}_3}\cup C_{\bm{y}_1}$. Therefore, $|C_{\bm{x}_3}\cap C_{\bm{x}_3'}|\geq 1$. We first show that $|C_{\bm{x}_3}\cap C_{\bm{x}_3'}|\geq 2$. Assume the contrary and without loss of generality, $C_{\bm{x}_3'}=\{1,2\}$ and $C_{\bm{x}_3}=\{1,3\}.$ Consider $\bm{x}_1$ and $\bm{x}_2$, there are two general cases: 
    \paragraph{Case A: $C_{\bm{x}_1}=C_{\bm{x}_2}$.}
    Observe that in this case, we have $C_{\bm{x}_1}=C_{\bm{x}_2}=\{2,3\}.$ Consider $\bm{y}_i\bm{x}_1\bm{x}_2\bm{x}_3$. Then $C_{\bm{y}_i}=\{1,2\}$ for all $i$. Hence $\bm{z}\bm{y}_1\bm{y}_2\bm{y}_3$ cannot belong to type $\mathfrak{A}.$

    \paragraph{Case B: $C_{\bm{x}_1}\neq C_{\bm{x}_2}$ and $\{C_{\bm{x}_1},C_{\bm{x}_2}\}=\{C_{\bm{x}_3'}, C_{\bm{x}_3}\}$.} Assume $C_{\bm{x}_1}=C_{\bm{x}_3}=\{1,3\}.$ Similar to the first case, $C_{\bm{y}_1}=C_{\bm{y}_2}=C_{\bm{y}_3}$ and then $\bm{z}\bm{y}_1\bm{y}_2\bm{y}_3$ cannot belong to type $\mathfrak{A}.$

    If $C_{\bm{x}_3}\neq C_{\bm{x}_3'}$, then one of them consists of three colors, say $C_{\bm{x}_3}=\{1,2,3\}$. Now, $C_{\bm{y}_i}$ is uniquely determined by $C_{\bm{x}_1}$ and $C_{\bm{x}_2}$, again leading to $C_{\bm{y}_1}=C_{\bm{y}_2}=C_{\bm{y}_3}$. Contradiction. 
\end{poc}
Consider the substructure $\mathcal{K}(\mathcal{T}';\bm{a}_1,\bm{a}_2).$ Since we glue a copy of $\mathcal{X}_1$ to each congruent copies of $\mathcal{T}'$ generated by $\{\bm{a}_i\}\cup\{\bm{b}_i\}\cup\{\bm{c}_{i,1},\bm{c}_{i,2}\}\cup\{\bm{z}_i,\bm{z}_i'\}\cup\{y_{i,j},y_{i,j}'\}$, by \cref{claim:no type B}, these are copies of type $\mathfrak{A}.$ Apply \cref{claim:no type A} to the twelve points $\bm{z}_i,\bm{y}_{i,1},\bm{y}_{i,2},\bm{y}_{i,3},\bm{b}_i,\bm{c}_{i,1},\bm{c}_{i,2}$ and $\bm{z}_i',\bm{y}_{i,1}',\bm{y}_{i,2}',\bm{y}_{i,3}',\bm{b}_{i+1},\bm{c}_{i,1},\bm{c}_{i,2}$, it follows that $C_{\bm{a}_1}= C_{\bm{a}_{2}}.$ Combine~\cref{claim:no type B} and~\cref{claim:no type A}, the substructure $\mathcal{K}(\mathcal{T}';\bm{a}_2,\bm{a}_3)$ also forces $C_{\bm{a}_2}= C_{\bm{a}_{3}}$, perishing the possibility for tetrahedron $\bm{a}_1\bm{a}_2\bm{a}_3\bm{a}_4$ belonging to type $\mathfrak{A}.$

However by the construction of \(\mathcal{X},\) we know that $\mathcal{X}_1(\bm{a}_1,\bm{a}_2,\bm{a}_3,\bm{a}_4)\subseteq\mathcal{X}$. Then by \cref{claim:no type B}, it follows that the tetrahedron $\bm{a}_1\bm{a}_2\bm{a}_3\bm{a}_4$ belongs to type $\mathfrak{A}$, which is a contradiction. This finishes the proof.
\end{proof}

\section{Concluding remarks}
In this paper we resolve the canonical Ramsey property for \emph{rectangles} and for \emph{triangles}. More precisely, for every rectangle \(\mathcal{R}\) there exists a dimension \(n_0=n_0(\mathcal{R})\), and for triangle \(\mathcal{T}\) one can take \(n_0=4\), such that for every \(r\in\mathbb{N}\) and every \(n\ge n_0\), every \(r\)-coloring of \(\mathbb{E}^n\) contains a monochromatic or a rainbow congruent copy of \(\mathcal{R}\) (resp. \(\mathcal{T}\)). The bound \(n_0=4\) for triangles is quite small; reducing it to \(3\) would be very interesting. For rectangles, the value \(n_0(\mathcal{R})\) produced by our argument is admittedly large, and optimizing its order remains an intriguing problem. In general, \(n_0(\mathcal{R})\) cannot be lowered to \(2\); see~\cite[Theorem~1.3]{2022arxivEGR} for a precise construction.

Our methods and results naturally point to a broader question. Recall that a configuration \(X\subseteq\mathbb{E}^n\) is called \emph{Ramsey} if for every \(r\) there exists a dimension \(n=n(X,r)\) with \(\mathbb{E}^n \overset{r}{\rightarrow} X\). Among currently known examples, configurations that exhibit the canonical Ramsey property such as rectangles, triangles, and certain simplices are also Ramsey. This motivates the following conjecture.

\begin{conjecture}
If \(\mathcal{S}\) is Ramsey, then \(\mathcal{S}\) exhibits the canonical Ramsey property.
\end{conjecture}

In particular, it is known that all rectangular sets and all simplices are Ramsey~\cite{1990JAMS}. Does the canonical Ramsey property extend to all \emph{rectangular sets}, that is, finite configurations contained in the vertex set of a rectangular box in \(\mathbb{E}^n\)? A positive answer would go well beyond the present case of full rectangles. Moreover, it remains a tantalizing open problem to determine whether {all} simplices have the canonical Ramsey property, an affirmative resolution would unify many of the special cases currently understood.

\section*{Acknowledgement}
Part of this work was initiated during the visits of Yijia Fang, Qian Xu, and Dilong Yang to the ECOPRO group at the IBS, as participants in the ECOPRO Summer Student Research Program. The authors acknowledge the organizers, especially Hong Liu, and the members of IBS for their support and hospitality.

\bibliographystyle{abbrv}
\bibliography{gallairamsey}

\end{document}